\documentclass[12pt]{amsart}

\calclayout

\usepackage{amsmath,amsthm,amssymb}
\usepackage{mathrsfs}

\usepackage{colonequals}
\usepackage{enumitem}
\usepackage{graphicx}
\usepackage[justification=centering]{caption}

\usepackage[colorlinks=true,allcolors=blue,pdftitle={GL with pinning in 2D}]{hyperref}

\usepackage[initials,nobysame,alphabetic,msc-links,backrefs]{amsrefs}
\usepackage{color}

\newtheorem{theorem}{Theorem}[section]
\newtheorem{lemma}{Lemma}[section]
\newtheorem{prop}{Proposition}[section]

\newtheorem{remark}{Remark}[section]

\numberwithin{equation}{section}

\DeclareMathOperator{\diver}{div}
\DeclareMathOperator{\curl}{curl}

\newcommand{\R}{\ensuremath{\mathbb{R}}}

\newcommand{\C}{\ensuremath{\mathbb{C}}}
\newcommand{\N}{\ensuremath{\mathbb{N}}}
\newcommand{\ip}[2]{\langle#1\hspace*{.5mm},#2\rangle}
\newcommand{\norm}[3][]{#1\left\|#2#1\right\|_{#3}}
\newcommand{\ep}{\varepsilon}
\newcommand{\ex}{\mathrm{ex}}

\renewcommand{\u}{\mathbf{u}}
\newcommand{\A}{\mathbf{A}}

\newcommand{\pmin}{\ensuremath{\rho_\varepsilon}}
\newcommand{\xfield}{\ensuremath{\xi_\varepsilon}}
\newcommand{\hfield}{\ensuremath{h_\varepsilon^0}}
\newcommand{\fen}{\ensuremath{F_{\varepsilon, \pmin}}}
\newcommand{\critfield}{\ensuremath{H^{\varepsilon}_{c_1}}}

\definecolor{purple}{rgb}{0.63, 0.36, 0.94}

\title[Meissner state for inhomogeneous superconductors]{On the Meissner state for type-II inhomogeneous superconductors}

\author{Mat\'{i}as D\'{i}az-Vera}
\address{Facultad de Matem\'aticas, Pontificia Universidad Cat\'olica de Chile, Vicu\~na Mackenna 4860, 7820436 Macul, Santiago, Chile}
\email{midiaz8@uc.cl}

\author{Carlos Rom\'{a}n}
\address{Facultad de Matem\'aticas e Instituto de Ingenier\'ia Matem\'atica y Computacional, Pontificia Universidad Cat\'olica de Chile, Vicu\~na Mackenna 4860, 7820436 Macul, Santiago, Chile}
\email{carlos.roman@uc.cl}

\date{\today}
\thanks{Funding information: ANID FONDECYT 1231593.}
\begin{document}
\begin{abstract} We consider extreme type-II superconductors modeled by the Ginzburg--Landau energy with a pinning term $a_\ep(x)$, which we assume to be a bounded measurable function such that $b\leq a_\ep(x)\leq 1$ for some constant $b>0$. A crucial feature of this type of  superconductors is the occurrence of vortices, which appear above the so-called first critical field $H_{c_1}$. In this paper we estimate this value and characterize the behavior of the Meissner solution, the unique vortexless configuration that globally minimizes the energy below $H_{c_1}$. In addition, we show that beyond this value, for applied fields whose strength is slightly below the so-called superheating field $H_{sh}$, there exists a unique Meissner-type solution that locally minimizes the energy.
\end{abstract}
\maketitle 

\section{Introduction}
    \subsection{The problem and brief state of the art}
    Superconductors are materials that can exhibit a complete loss of electrical resistance when cooled below a critical temperature (typically very low). Superconductivity was discovered by Kamerlingh Onnes in 1911. The two most striking features of it are the possibility of \emph{permanent superconducting currents} and the expulsion of applied magnetic fields, which in turn leads to \emph{superconducting levitation}.
    In type-II superconductors, the normal and superconductivity phases may coexist in the material. Indeed, a key physical feature of type-II superconductivity is the occurrence of \emph{vortices} (similar to those in fluid mechanics but quantized), in the presence of an applied magnetic field. In these regions, the external field penetrates the material, and the superconductivity is lost. 
    
    These vortices may move because of internal interactions and external forces. Their motion generates an electric field that dissipates energy and, in turn, generates an electrical resistance, thus losing superconductivity in the material. One way to control the motion of the vortices is to introduce inhomogeneities into the material, which provide \emph{pinning sites} for the vortices. We refer to \cites{chapman-richardson, chapman-du-gunzburguer-variable-thickness, chapman-richardson-pinning,ding-du} and the references therein for more details of the physics of pinned superconductors.

    \medskip
    The behavior of pinned superconductors is modeled by the famous Ginzburg--Landau model of superconductivity (with pinning), which is defined by 
    \begin{equation}\label{GLenergy}
    GL_\varepsilon(u,A) \colonequals \frac{1}{2} \int_\Omega |\nabla_A u|^2 + \frac{\left(a_{\varepsilon}(x)-|u|^2\right)^2}{2\varepsilon^2}+|h-h_{\ex}|^2 .
    \end{equation}
    Here
    \begin{itemize}
        \item $\Omega\subset \R^2$ is a smooth, bounded, and simply connected domain.
        \item $u:\Omega \to \C$ is the order parameter. Its squared modulus (the density of Cooper pairs of superconducting electrons in the Bardeen--Cooper--Schrieffer (BCS) quantum theory \cite{BCS}) indicates the local state (normal or superconducting) of the material.
        \item $A:\Omega \to \R^2$ is the electromagnetic vector potential of the induced magnetic field $h= \operatorname{curl} A\colonequals \partial_{x_1}A_2-\partial_{x_2}A_1$.
        \item $\nabla_A$ denotes the covariant gradient $\nabla - iA$.
        \item $h_{\ex}>0$ is a constant that represents the intensity of the external magnetic field in the direction perpendicular to $\Omega$.
        \item $\ep>0$ is the inverse of the \emph{Ginzburg--Landau parameter} usually denoted $\kappa$, a nondimensional parameter depending only on the material. We will be interested in the regime of small $\ep$, corresponding to extreme type-II superconductors. 
        \item $a_{\varepsilon}$ is a function that accounts for inhomogeneities in the material. We will assume that $a_\ep \in L^\infty(\Omega)$ and that it takes values in $[b,1]$, where $b\in (0,1)$ is a constant independent of $\varepsilon$. The regions where $a_\varepsilon = 1$ correspond to sites without inhomogeneities (we also say that there is no pinning in these regions). 
    \end{itemize}
    
    Both the mathematics and physics literature on the effect of pinning in the Ginzburg--Landau model of superconductivity is quite extensive. Without aiming for a full bibliographic review, we next mention a few of the pinning-type functions that have been considered in the literature:
    \begin{itemize}
        \item Models where $a_\varepsilon$ is smooth were studied in \cite{aftalion2000pinning}, where, in addition to our assumption, $a_\varepsilon$ homogenizes in the sense of $H$-convergence as $\varepsilon \to 0$, in \cite{andre-bauman-phillips}, where the pinning term has a finite number of zeroes, and in \cite{aftalion-alama-bronsard}, where the pinning term is radial and is allowed to be negative.
        
        \item Models where $a_\varepsilon$ is a step function taking only two values (say $b$ and $1$) were studied in \cites{lassoued-mironescu,DS2,DS3,DS4,DS5} in the pinned Ginzburg--Landau model without external magnetic field, in \cites{kachmar,kachmar2}, where $\Omega = B(0,1)$ and $a_{\varepsilon}^{-1}(b)$ is an annulus, and in \cite{santos2019magnetic}, where the region $a_{\varepsilon}^{-1}(b)$ is periodic and shrinks as $\varepsilon \to 0$. 

        \item A model where $a_\varepsilon$ is an oscillating periodic function under the effect of a random ergodic stationary action was studied in \cite{santos2022ginzburglandau}.
    \end{itemize}
    It is worth pointing out that, in some of these models, we expect the minima of $a_{\varepsilon}$ to be pinning sites for the vortices. There is computational evidence in \cites{du-gunzburger-peterson,chapman-du-gunzburger} and a proof in \cite{sigal-ting} for $h_{\ex} = 0$ and a sufficiently small and smooth pinning term.

    \medskip
    The homogeneous case $a_\varepsilon \equiv 1$ in $\Omega$ corresponds to the celebrated Ginzburg--Landau functional proposed by Ginzburg and Landau \cite{GinLan}. This model has been extensively studied by analysts after the seminal work \cite{bethuel-brezis-helein} on the functional without magnetic field ($h_{\ex} = A = 0$). We refer to the classical book \cite{libro-ss} for an extensive mathematical review of the model with magnetic field and to \cite{introduction-superconductivity} for a more physics-oriented study of superconductivity and vortex pinning.

    \medskip
    An important feature of the Ginzburg--Landau model is that all the physically meaningful quantities are gauge-invariant, which means that they are preserved under the gauge transformation
    $$(u,A) \to (ue^{i\phi}, A+\nabla \phi).$$
    The energy $GL_\varepsilon$ and the free energy with (or without) weight $\eta_\varepsilon \colon \Omega \to \R$ defined via
    \begin{equation}\label{free-energy-weight}F_{\varepsilon, \eta_\varepsilon}(u,A) = \frac{1}{2}\int_\Omega \eta_\varepsilon^2 |\nabla_A u|^2 + \eta_{\varepsilon}^4 \frac{(1-|u|^2)^2}{2\varepsilon^2}+ |\curl A|^2 \end{equation}
    are gauge-invariant, as well as the induced magnetic field $h$, the density of Cooper-pairs of electrons $|u|$, the superconducting current $\ip{iu}{\nabla_A u}$ and the vorticity measure, defined by
    $$\mu(u,A) = \curl \ip{iu}{\nabla_A u} + h,$$
    where $\ip{\cdot}{\cdot}$ is the scalar product of $\C$ identified with $\R^2$, that is, $\ip{z}{w} = \frac{z\overline{w}+\overline{z}w}{2}$. We will denote by $F_\ep(u,A)$ the free energy without weight (that is, when $\eta_\ep\equiv 1$ in $\Omega$). One particular gauge-choice, so called \emph{Coulomb gauge}, is the one for which
    \begin{equation}\label{coulomb gauge}
            \left\{
            \begin{array}{rcll}
            -\diver A &=& 0 &\mathrm{in}\ \Omega\\
            A \cdot \nu &=& 0&\mathrm{on}\ \partial \Omega.
            \end{array}
            \right.
        \end{equation}
        
    \medskip

    It is well known that as the intensity of the external magnetic field $h_\ex$ is tuned, type-II superconductors undergo several phase transitions. There are three main critical values $H_{c_1}, H_{c_2}$ and $H_{c_3}$ for $h_{\ex}$, called (main) \emph{critical fields}, where phase transitions occur: When $h_{\ex} < H_{c_1}$, the material is everywhere in its superconducting phase, that is, $|u|$ is uniformly close to $1$, and the applied field is expelled by the material due to the occurrence of supercurrents near $\partial\Omega$. This phenomenon is known as \emph{Meissner effect}. When $h_{\ex} \geq H_{c_1}$, the external magnetic field penetrates the material and vortices start to appear, and as $h_{\ex}$ increases, so does the number of vortices. Near $H_{c_2}$, superconductivity is lost in the bulk of the material and when $H_{c_2} < h_{\ex} < H_{c_3}$, superconductivity remains only near the boundary. When $h_{\ex} > H_{c_3}$, superconductivity is lost, as the external magnetic field has completely penetrated the material.

    \subsection{Main results}
    A main purpose of this paper is to provide a precise approximation of the \emph{Meissner state} (or configuration), that is, the unique (modulo gauge-invariance) solution of the Ginzburg--Landau equations without vortices, which, in turn, allows us for providing an estimate for the main order of the first critical field, which is of order $O(|\log\ep|)$, and to show that the Meissner solution is stable for values of the intensity of the applied field close to the so-called \emph{superheating} field, which is of order $O(\ep^{-1})$.

    Before precisely stating our main results, let us introduce the configuration 
    \begin{equation}\label{Meissneraprox}
        (\rho_\ep,h_\ex A_\ep^0)\colonequals \left(\pmin, h_{\ex} \frac{-\nabla^\perp \xfield}{\pmin^2} \right),
    \end{equation}
    which corresponds to our approximation of the Meissner state. Here
    \begin{itemize}[]
    \item $\pmin$ is the unique positive real-valued minimizer in $H^1(\Omega,\C)$ of the pinned Ginzburg--Landau energy functional without magnetic field
    \begin{equation}\label{energyLM}
    E_\varepsilon(u) \colonequals \frac{1}{2}\int_\Omega |\nabla u|^2 + \frac{(a_\varepsilon-|u|^2)^2}{2\varepsilon^2}.\end{equation}
    It satisfies the Euler--Lagrange equation
    \begin{equation}\label{pde rho}
            \left\{
            \begin{array}{rcll}
            -\Delta \pmin &=& \dfrac{\pmin(a_\varepsilon-\pmin^2)}{\varepsilon^2} &\mathrm{in}\ \Omega\\
            \dfrac{\partial \pmin}{\partial \nu} &=& 0&\mathrm{on}\ \partial \Omega,
            \end{array}
            \right.
        \end{equation}
    which combined with the maximum principle, yields $\sqrt{b}\leq \rho_\ep\leq 1$. In addition, by taking a constant as an energy competitor, we find the a priori bound 
    \begin{equation}\label{rho energy bound}
        E_\varepsilon(\pmin) \leq  \frac{C}{\varepsilon^2}.
    \end{equation}
    \item $\xfield$ is the unique solution in $H_0^1(\Omega)$ to
    \begin{equation}\label{pde xi}
    \left\{
    \begin{array}{rccl}
        -\diver \left(\dfrac{\nabla \xfield}{\pmin^2}\right)+\xfield&=&1&\mbox{in }\Omega\\
        \xfield&=&0&\mbox{on }\partial\Omega.
    \end{array}
    \right.
    \end{equation}
    Since $\sqrt{b}\leq \rho_\ep\leq 1$, the maximum principle yields $0\leq \xi_\ep\leq 1$. Furthermore, observe that the differential operator associated with $\xfield$ is uniformly elliptic. The classical result of Meyers \cite{meyers}*{Theorem 1.1} thus ensures that $\xfield \in W^{1,p_0}(\Omega)$, for some $p_0 > 2$ that does not depend on $\varepsilon$ and, more importantly, that there exists $C = C(\Omega,b)>0$ such that 
    \begin{equation*}%
        \norm{\nabla \xfield}{L^{p_0}(\Omega)} \leq C.
    \end{equation*}
    Quite surprisingly, the special structure of the elliptic problem \eqref{pde xi} implies that the previous estimate holds with $p_0=\infty$ (see Proposition \ref{prop:linftyxi}), which will play a fundamental role throughout the article.
    \end{itemize}

    A crucial result concerning this special configuration is the following energy splitting.
    \begin{prop}\label{prop-splitting} Given any configuration $(\mathbf{u,A}) \in H^1(\Omega,\C)  \times H^1(\Omega,\R^2)$, letting $(u,A)$ be defined through the relation $(\mathbf{u,A}) = \left(\rho_\ep u, A + h_{\ex}A_\ep^0\right)$, we have
    \begin{equation}\label{Split}
     GL_\varepsilon \left(\mathbf{u,A} \right) =  GL_\ep\left(\rho_\ep,h_\ex A_\ep^0\right) + F_{\varepsilon,\pmin}(u,A) - h_{\ex} \int_\Omega \mu(u,A) \xfield + R_0,
    \end{equation}
    where 
    \begin{equation}\label{defR0}
    R_0\colonequals \frac{h_\ex^2}2\int_\Omega \frac{|\nabla \xi_\ep|^2}{\rho_\ep^2}\left(|u|^2-1\right).
    \end{equation}
    \end{prop}
    Let us remark that $R_0=R_0(\ep)$ is a term that is negligible in the regime of $h_{\ex}$ that we are interested in. 
    The first term in the RHS of \eqref{Split} captures, with high precision, the minimal energy among configurations that do not have vortices. More precisely, we have that
    $$
    GL_\ep\left(\rho_\ep,h_\ex A_\ep^0\right)=E_\ep(\rho_\ep)+h_\ex^2 J_\ep(A_\ep^0),\quad \mbox{where }J_\ep(A_\ep^0)=\frac12\int_\Omega \rho_\ep^2|A_\ep^0|^2+|\curl A_\ep^0-1|^2.
    $$
    Notice that $E_\ep(\rho_\ep)$ corresponds to the cost ``enforced'' by the potential term in \eqref{GLenergy}, which is captured by the fact that $\rho_\ep^2\approx a_\ep$, while the later term $h_\ex^2 J_\ep(A_\ep^0)$ corresponds to the energy cost produced by the presence of the external field. 

    Let us also observe that, since $u_\ep^0=\rho_\ep\geq \sqrt{b}>0$ (see Proposition \ref{proprho_ep}) $\mathbf u$ and $u$ have the same vortices and $\mu(\mathbf u,\mathbf A)\approx \mu(u,A)$. For this reason, the second term in the RHS of \eqref{Split} can be thought of as the energetic cost of the vortices, while the third term is the magnetic gain due to the vortices. Hence, the occurrence of vortices strongly depends on the sign of $F_{\varepsilon,\pmin}(u,A) - h_{\ex} \int_\Omega \mu(u,A) \xfield$. We refer the reader to Section \ref{heuristicHc1} for a more detailed (heuristic) discussion.

    It is worth mentioning that this splitting was strongly prompted by \cites{BetRiv,lassoued-mironescu,libro-ss,kachmar}.

    \medskip
    We are now ready to state our first result, which concerns an estimate of the main order of the first critical field $H_{c_1}$. Recall that the first critical field is (rigorously) defined by the fact that below this value, global minimizers of \eqref{GLenergy} do not have vortices, while they do for applied fields whose strength is greater than $H_{c_1}$. Letting
    \begin{equation}\label{mainHc1}
    \critfield \colonequals  \frac{|\log\ep|}{2 \max_{\Omega} \psi_\varepsilon},
    \end{equation}
    where $\psi_\varepsilon = \frac{\xfield}{\pmin^2}$, we have the following result.
    \begin{theorem}\label{crit field lower bound}
    There exist $\varepsilon_0 > 0$ and $K_0 > 0$ such that, for any $\varepsilon< \varepsilon_0$ and any $h_{\ex} \leq \critfield - K_0 \log|\log \varepsilon|$, the global minimizers $(\mathbf{u,A})$ of $GL_\varepsilon$ in $H^1(\Omega,\C)\times H^1(\Omega,\R^2)$ are vortexless configurations such that, letting $(u,A)=(\rho_\ep^{-1}\u,\A-h_\ex A_\ep^0)$, as $\ep\to 0$, we have
    \begin{enumerate}[leftmargin=*,label={\normalfont (\arabic*)}]
        \item $\norm{1-|u|}{L^\infty(\Omega,\C)} = o(1)$.
        \item $\norm{\mu(u,A)}{\left(C_0^{0,1}(\Omega)\right)^*} = o(1)$.
        \item $|GL_\ep(\u,\A)-GL_\ep(\rho_\ep,h_\ex A_\ep^0)|=o(1)$.
    \end{enumerate}
    \end{theorem}
    This result characterizes the behavior of global minimizers below $H_{c_1}^\ep$. The next result provides a characterization above this value. 
    \begin{theorem}\label{crit field upper bound}
    Assume $[\rho_\ep^2]_{C^{0,\alpha}(\Omega)} \leq |\log \ep|^m$ for some $m>0$ and $\alpha \in (0,1]$, where $[\ \cdot\ ]_{C^{0,\alpha}(\Omega)}$ denotes the Hölder seminorm. Then, there exist $\varepsilon_0 > 0$ and $K^0 > 0$ such that, for any $\varepsilon< \varepsilon_0$ and any $h_\ex$ such that $\critfield + K^0 \log|\log \varepsilon|\leq h_\ex\leq |\log \ep|^N$ for some $N\geq 2$, the global minimizers $(\mathbf{u,A})$ of $GL_\varepsilon$ in $H^1(\Omega,\C)\times H^1(\Omega,\R^2)$ do have vortices.
    \end{theorem}
    Thus, under the assumptions of the previous two theorems, we conclude that 
    $$
    H_{c_1}=H_{c_1}^\ep+O(\log|\log\ep|).
    $$
    This in particular generalizes the estimate on the first critical field found in \cite{kachmar}, where, as explained above, $\Omega$ is a ball and the pinning term is radial. It is also worth remarking that, without further assumptions on $a_\ep$, it is not possible to reduce the error term in the estimate. In the homogeneous case $a_\ep\equiv 1$ in $\Omega$, it is well known (see \cites{SS-Hc1,libro-ss}) that $H_{c_1}=C(\Omega)|\log \ep|+o(1)$ as $\ep\to 0$. However, in \cite{santos2019magnetic} Dos Santos showed that, under the assumption on $a_\ep$ explained above, the expansion of the first critical field contains a term of the form $C(\Omega,b)\log|\log\ep|$. 

    \begin{remark}\label{remarkmaxpsi}
    Since $\rho_\ep^2 \in [b,1]$ and $\xi_\ep\in [0,1]$, we have
    $$
    0<\max_\Omega \xi_\ep \leq \max_\Omega \psi_\ep\leq b^{-1} \max_\Omega \xi_\ep\leq b^{-1}<+\infty.
    $$
    Moreover, it holds that $\liminf_{\ep\to 0}\max_\Omega \xi_\ep>0$ (see Proposition \ref{prop:liminf}). Hence, just as in the homogeneous case, $H_{c_1}^\ep=O(|\log\ep|)$. We believe that it would be interesting to investigate whether $\max_{\Omega}\psi_\ep$ converges, as $\ep\to 0$, to some special constant depending on $\Omega$ and $b$ if one considers a model where $a_\ep$ homogenizes as $\ep \to 0$, that is, in the spirit of the one considered in \cite{aftalion2000pinning}.
    \end{remark}

    \medskip 
    Our next results go beyond the first critical field. They show that, as in the homogeneous case $a_\ep\equiv 1$ in $\Omega$, the Meissner state beyond $H_{c_1}$ continues to be a local minimizer of the energy, even for applied fields with intensity close to $O(\ep^{-1})$. We begin by presenting an existence result.

    \begin{theorem}\label{existence-above-fcf}
    Let $\alpha \in \left(0,\frac{1}{2}\right)$. There exists $\varepsilon_0 > 0$ such that, for any $\varepsilon < \varepsilon_0$ and $h_{\ex} < \varepsilon^{-\alpha}$, there exists a vortexless local minimizer $\left(\mathbf{u, A}\right)$ for $GL_\varepsilon$. Moreover, letting $(u,A)=(\rho_\ep^{-1}\u,\A-h_\ex A_\ep^0)$ as $\ep\to 0$, we have
    \begin{enumerate}[leftmargin=*,label={\normalfont (\arabic*)}]
        \item $\norm{\mu(u,A)}{\left(C^{0,1}_0(\Omega)\right)^*} = o(1)$.
        \item $|GL(\u,\A)-GL(\rho_\ep,h_\ex A_\ep ^0)|=o(1)$.
        \item $\norm{1-|u|}{L^\infty(\Omega)} = o(1).$
    \end{enumerate}
    Furthermore, if $(\u, \A)$ is in the Coulomb gauge, it holds that
    \begin{enumerate}[leftmargin=*,label={\normalfont (\arabic*)},resume]
        \item The configuration $(u,A)$ satisfies
        $$\inf_{\theta \in [0,2\pi]} \norm{u - e^{i\theta}}{H^1(\Omega)} + \norm{A}{H^1(\Omega)} = o(1).$$
        \item The configuration $(\mathbf{u,A})$ satisfies
        $$\norm{\mathbf{A} - h_{\ex} A_\ep^0}{H^1(\Omega)} = o(1).$$
        Moreover, if $\norm{\nabla \pmin}{L^2(\Omega)} < \varepsilon^{-\gamma}$ for some $\gamma < 1-2\alpha$, we have for any $r \in [1,2)$ 
        $$\inf_{\theta \in [0,2\pi]} \norm{\mathbf{u} - \pmin e^{i\theta}}{W^{1,r}(\Omega)} = o(1).$$
    \end{enumerate}
\end{theorem}
Let us emphasize that this result gives a precise characterization of the behavior of the local minimizer. As a matter of fact, it essentially shows that our approximation of the Meissner state is almost a solution of the Ginzburg--Landau equations. As far as we know, analogous results have only been established for the homogeneous Ginzburg--Landau functional; see \cite{sylvia-arma} (in the 2D case) and \cite{Rom-CMP} (in the 3D case).

\medskip
Our last result concerns the uniqueness, up to a gauge transformation, of locally minimizing vortexless configurations. %

\begin{theorem}\label{uniqueness-above-fcf}
    Assume $E_\ep(\pmin) \ll \frac{1}{\ep^2}$. Let $\alpha \in (0,1)$ and $\beta > 0$. There exists $\varepsilon_0 > 0$ such that for any $\varepsilon < \varepsilon_0$, if $h_{\ex} \leq \varepsilon^{-\alpha}$ then a pair $(\u,\A) = (\pmin u, A + h_\ex A^0_\ep)$ which locally minimizes $GL_\ep$ in $H^1(\Omega,\C)\times H^1(\Omega,\R^2)$ and satisfies $\fen(u,A) < \varepsilon^\beta$, is unique up to a gauge transformation.
\end{theorem}

\begin{remark}
The hypotheses of this theorem are verified by the vortexless local minimizer found in Theorem \ref{existence-above-fcf}. More precisely, given $\alpha\in \left(0,\frac12\right)$, the vortexless local minimizer $(\u,\A)$ given by Theorem \ref{existence-above-fcf} is such that $\fen(u,A) < \varepsilon^\beta$, for some constant $\beta>0$. 
\end{remark}

Hence, in summary, we see that for $h_\ex \leq H_{c_1}^\ep-K\log|\log \ep|$, the unique global minimizer of \eqref{GLenergy} (up to a gauge transformation) is a vortexless configuration that looks very similar to \eqref{Meissneraprox}. Beyond this value, at least up to $h_\ex=o(\ep^{-\frac12})$, \eqref{GLenergy} admits a unique vortexless local minimizer with the same behavior. Therefore, exactly as observed in the homogeneous Ginzburg--Landau functional, since this branch of vortexless solutions remains stable, in the process of continuously rising $h_\ex$, vortices should appear at a critical value of $h_\ex$ called the \emph{superheating field} $H_{sh}$ instead of when getting to the first critical field $H_{c_1}=O(|\log\ep|)$. When reaching $H_{sh}$, the Meissner configuration becomes unstable, allowing for the occurrence of vortices. In the homogeneous case $H_{sh}=O(\ep^{-1})$; see, for instance, the classical works \cites{superheating2,superheating}.

\medskip
Let us finally mention that two classical tools in the analysis of Ginzburg--Landau type energies play a crucial role in this paper: the vortex ball construction and the vorticity estimate. In this paper, we obtain a new version of the former that might be of independent interest, which generalizes the vortex ball construction method provided in \cite{mass-displacement} to the case of the weighted Ginzburg--Landau free energy \eqref{free-energy-weight}; see Proposition \ref{ball lower bound}. 

\subsection*{Outline of the paper}
The rest of the paper is organized as follows. In Section \ref{splitting section}, we provide some preliminary results and a proof of the energy splitting \eqref{Split}. In Section \ref{first cf}, we heuristically derive $\critfield$ and prove Theorem \ref{crit field lower bound} and Theorem \ref{crit field upper bound}. In Section \ref{above fcf}, we prove Theorem \ref{existence-above-fcf} and Theorem \ref{uniqueness-above-fcf}. Finally, in Appendix \ref{ball construction chapter}, we provide the new version of the vortex ball construction method for a weighted Ginzburg--Landau energy.

\noindent
\\
{\bf Acknowledgments.} This work was partially funded by ANID FONDECYT 1231593.

\section{Preliminaries and Energy Splitting}\label{splitting section}
    \subsection{The (weighted) Ginzburg--Landau equations.}
    The Euler--Lagrange equations associated to $GL_\varepsilon$ are
    \begin{equation}\label{euler lagrange gl}
        \left\{
        \begin{array}{rcll}
        -(\nabla_A)^2u &=& \dfrac{u(a_\varepsilon - |u|^2)}{\varepsilon^2}&\mbox{in }\Omega \\
        -\nabla^\perp h &=& \ip{iu}{\nabla_A u} &\mbox{in }\Omega\\
            h&=&h_{\ex}&\mbox{on }\partial \Omega\\
            \nabla_A u \cdot \nu &=& 0 &\mbox{on }\partial \Omega,
        \end{array}
        \right.
    \end{equation}
    where $(\nabla_A)^2 = (\operatorname{div}-iA)(\nabla_A)$ and $\nu$ is the unit normal vector pointing outward from $\Omega$.

    Observe that the maximum principle implies that any solution to \eqref{euler lagrange gl} satisfies 
    \begin{equation}\label{maxprinc|u|}
    |u|^2 \leq \max_{\Omega} a_{\varepsilon} \leq 1
    \end{equation}
    This can be proved following exactly the same argument used in the case of the classical homogeneous Ginzburg--Landau energy (see for instance \cite{libro-ss}*{Proposition 3.8}). %
    
    The Ginzburg--Landau equations \eqref{euler lagrange gl} are invariant under gauge transformations. Therefore, any solution of \eqref{euler lagrange gl} can be gauge-transformed into a solution $(u,A)$ in the Coulomb gauge (see for instance \cite{libro-ss}*{Proposition 3.2}). One of the advantages of this particular choice of gauge lies in some elliptic regularity estimates, as we shall see later on.

    \subsection{The function \texorpdfstring{$\pmin$}{ρ}}
    The function $\pmin$ was firstly introduced by Lassoued and Mironescu in \cite{lassoued-mironescu}. Basically, it corresponds to a regularized version of $\sqrt{a_\varepsilon}$. A key tool developed in \cite{lassoued-mironescu} is the following decoupling of the energy $E_\ep$ (recall \eqref{energyLM}) 
        \begin{equation}\label{lm decoupling}
            E_\varepsilon(\pmin u) = E_\varepsilon(\pmin) + \frac{1}{2} \int_\Omega \pmin^2 |\nabla u|^2 +\pmin^4 \frac{(1 - |u|^2)^2}{2\varepsilon^2}.
        \end{equation}
        This in particular means that one can study the effect of pinning in terms of a weighted Ginzburg--Landau energy with homogeneous potential term $(1-|u|^2)^2$.

        The previous decoupling of the energy even holds if one replaces the gradient term $|\nabla u|$ by the covariant derivative $|\nabla_A u|$ in \eqref{energyLM}, that is, if one considers the energy functional
        $$E_\varepsilon(u,A)\colonequals  \frac{1}{2} \int_\Omega |\nabla_A u|^2 + \frac{(a_\varepsilon - |u|^2)^2}{2\varepsilon^2}.$$
        The following is a classical result, but we provide a proof for the sake of completeness.
        \begin{lemma}\label{LM decoupling}
            For any $(u,A) \in H^1(\Omega,\C) \times H^1(\Omega,\R^2)$, we have
            \begin{equation}\label{LM eq}
                E_\varepsilon(\pmin u,A)= E_\varepsilon(\pmin) + \frac{1}{2} \int_\Omega \pmin^2 |\nabla_A u|^2 + \pmin^4 \frac{(1-|u|^2)^2}{2\varepsilon}.
            \end{equation}
         \end{lemma}
    
        \begin{proof}
            Expanding the square on $|\nabla_A u|^2$ we have
            $$|\nabla_A u|^2 = |\nabla (\pmin u)|^2 + \pmin^2 |A|^2|u|^2 - 2\pmin \ip{\nabla(\pmin u)}{iAu}.$$
            Combining with \eqref{lm decoupling}, we find
            \begin{multline*}
                E_\varepsilon(\pmin u) = E_\varepsilon(\pmin) + \frac{1}{2}\int_\Omega \pmin^2|\nabla u|^2 + \pmin^4 \frac{(1-|u|^2)^2}{2\varepsilon^2}\\ + \pmin^2 |A|^2 |u|^2 - 2\pmin (\ip{\pmin \nabla u}{iAu}+\ip{u \nabla \pmin}{iAu}).
            \end{multline*}
            Since $\nabla \pmin$ and $A$ are real-valued vector fields, $\ip{u \nabla \pmin}{iAu} = 0$. Thus, the RHS is equal to
            \begin{multline*}
                E_\varepsilon (\pmin) + \frac{1}{2} \int_\Omega \pmin^2(|\nabla u|^2 + |A|^2|u|^2 - 2 \ip{\nabla u}{iAu}) + \pmin^4 \frac{(1-|u|^2)^2}{2\varepsilon^2} \\= E_\varepsilon (\pmin) + \frac{1}{2} \int_\Omega \pmin^2|\nabla_A u| + \pmin^4 \frac{(1-|u|^2)^2}{2\varepsilon^2}.
            \end{multline*}
        \end{proof}

        Let us now state some regularity properties of $\rho_\ep$.
    \begin{prop}\label{proprho_ep}
        We have $\sqrt{b} \leq \rho_\ep\leq 1$ and $\norm{\nabla \pmin}{L^\infty(\Omega)} \leq \frac{C}{\varepsilon}$ for some $C>0$.
    \end{prop}

    \begin{proof}
        The Euler--Lagrange equation associated with the energy functional \eqref{energyLM} is \eqref{pde rho}.
        Testing this equation against $\max\{\pmin(x),1\}$, we are led to
        $$0 \leq \int_{\{\pmin>1\}} |\nabla \pmin|^2 = \int_{\{\pmin>1\}} \pmin^2 (a_\varepsilon-\pmin^2).$$
        Since $\pmin^2(a_\varepsilon - \pmin^2) < 0$ when $\pmin > 1$, we deduce that $|\{\pmin > 1\}| = 0$, which means $\pmin \leq 1$. By testing against $\min\left\{\pmin(x),\sqrt{b}\right\}$, we obtain the other inequality.

        The estimate on the gradient follows from the Gagliardo--Nirenberg type inequality for functions $u \in H^2(\Omega)$ such that $\frac{\partial u}{\partial \nu} = 0$ on $\partial\Omega$ (see \cite{santos2019magnetic}*{Lemma 3.2})     
        \begin{equation}\label{neumann condition grad estimate}
            \norm{\nabla u}{L^\infty(\Omega)}^2 \leq C\left(\norm{\Delta u}{L^\infty(\Omega)} +\norm{u}{L
        ^\infty(\Omega)}\right) \norm{u}{L^\infty(\Omega)}.
        \end{equation}
        Indeed, since $\norm{\pmin}{L^\infty(\Omega)}, \norm{a_\varepsilon}{L^\infty(\Omega)} \leq 1$, from \eqref{pde rho} we obtain that $\norm{\Delta \pmin}{L^\infty(\Omega)} \leq \frac{C}{\varepsilon^2}$, which leads to
        $$\norm{\nabla \pmin}{L^\infty(\Omega)} \leq \frac{C}{\varepsilon}.$$
    \end{proof}

    \begin{prop}
        Suppose $a_\varepsilon \in H^1(\Omega)$. It holds that:
        \begin{enumerate}[leftmargin=*,label={\normalfont (\arabic*)}]
            \item There exists a constant $C>0$ such that $\norm{\nabla \pmin}{L^2(\Omega)} \leq C \norm{\nabla a_\varepsilon}{L^2(\Omega)}$.
            \item For $\alpha \in (0,1)$, let $X_\alpha \colonequals \{x \in \Omega \colon |a_\varepsilon(x) - \pmin(x)^2| > \varepsilon^{\alpha}\}$. Then, for some $C>0$, we have $$|X_\alpha| < C \norm{\nabla a_\varepsilon}{L^2(\Omega)}^2 \varepsilon^{2(1-\alpha)}.$$
        \end{enumerate}
    \end{prop}
    \begin{proof}
            On the one hand, since $\pmin$ is a minimizer in $H^1(\Omega,\C)$, we have
            $$\norm{\nabla \pmin}{L^2(\Omega)} ^2 \leq E_\varepsilon(\pmin) \leq E_\varepsilon(\sqrt{a_\varepsilon}) = \frac{1}{2}\norm{\nabla \sqrt{a_\varepsilon}}{L^2(\Omega)}^2 \leq C \norm{\nabla a_\varepsilon}{L^2(\Omega)}^2.$$
            On the other hand,
            $$C\norm{\nabla a_\varepsilon}{L^2(\Omega)}^2 \geq E_\varepsilon(\pmin) \geq \int_{X_\alpha} \frac{(a_\varepsilon -\pmin^2)}{2\varepsilon^2} > \frac{|X_\alpha| \varepsilon^{2\alpha}}{\varepsilon^2} .$$
            Hence
            $$|X_\alpha| < C\norm{\nabla a_\varepsilon}{L^2(\Omega)}^2 \varepsilon^{2(1-\alpha)}.$$
    \end{proof}
    \begin{remark}
        Although we will not use this result in this paper, we present it to the reader to better understand the role of $\rho_\ep$. This result shows that, when $a_\ep$ is regular enough, $\pmin$ is a very close approximation of $\sqrt{a_\varepsilon}$, except for a very small set. This small set is expected to be located near the boundary and near the discontinuity regions of $a_\varepsilon$; see \cites{aftalion-alama-bronsard,santos2019magnetic} for some specific models.
    \end{remark}

    \subsection{Estimates for critical points in the Coulomb gauge}
        Let us present some estimates for configurations $(u,A)$ in the Coulomb-gauge, that is, when $A$ satisfies \eqref{coulomb gauge}. From \cite{libro-ss}*{Proposition 3.3}, we have
        \begin{equation}\label{coulomb h1 estimate}
            \norm{A}{H^1(\Omega)} \leq C \norm{\curl A}{L^2(\Omega)},
        \end{equation}
        and 
        \begin{equation}\label{coulomb h2 estimate}
            \norm{A}{H^2(\Omega)} \leq C \norm{\curl A}{H^1(\Omega)},
        \end{equation}
        where $C>0$ depends only on $\Omega$. These estimates play a crucial role on obtaining better regularity results for solutions of \eqref{euler lagrange gl}. In particular, we have the following three results.
        \begin{prop}
            Let $(\u,A) = (\pmin u,A)$ be a solution of \eqref{euler lagrange gl}, where $A$ satisfies \eqref{coulomb gauge}. Then
            \begin{equation}\label{coulomb linfty estimate}
                \norm{A}{L^\infty(\Omega)} \leq C(E_\ep(\pmin) + \fen(u,A))^{\frac{1}{2}},
            \end{equation}
            where $C>0$ depends only on $\Omega$.
        \end{prop}
        \begin{proof}
            From the second equation in \eqref{euler lagrange gl} and the Cauchy--Schwarz inequality, we deduce that
            $$\norm{\nabla \curl A}{L^2(\Omega)} = \norm{\ip{i \u}{\nabla_A \u}}{L^2(\Omega)} \leq \norm{\u}{L^2(\Omega)} \norm{\nabla_A \u}{L^2(\Omega)}.$$
            Since $|\u| \leq 1$ (recall \eqref{maxprinc|u|}), it follows that
            $$\norm{\nabla \curl A}{L^2(\Omega)}^2 \leq  \norm{\nabla_A \u}{L^2(\Omega)}^2 \leq  C E_\ep(\u,A).$$
            The decoupling \eqref{LM eq} then yields
            $$\norm{\nabla \curl A}{L^2(\Omega)}^2 \leq C \left(E_\ep(\pmin) + \fen(u,A) \right).$$
            Moreover, since $\norm{\curl A}{L^2(\Omega)}^2 \leq 2\fen(u,A)$, we deduce that
            $$\norm{\curl A}{H^1(\Omega)}^2 \leq C(E_\ep(\pmin) + \fen(u,A)).$$
            Finally, by \eqref{coulomb h2 estimate} and Sobolev embedding, we obtain \eqref{coulomb linfty estimate}.
        \end{proof}
        The hypotheses of our main result will allow us to control the RHS of \eqref{coulomb linfty estimate} by $\frac{C}{\ep}$, for a constant $C>0$ independent of $\ep$. This in turn allows us to obtain the following estimate.
        \begin{prop}
            Let $(\u,A) = (\pmin u,A)$ be a solution of \eqref{euler lagrange gl}, where $A$ satisfies \eqref{coulomb gauge} and $\|A\|_{L^\infty(\Omega)} \leq \frac{\tilde C}{\ep}$ for some $\tilde C>0$ not depending on $\ep$. Then
            \begin{equation}\label{coulomb grad estimate}
            \norm{\nabla \u}{L^\infty(\Omega)}\leq \frac{C}{\ep}\quad \mbox{and}\quad \norm{\nabla u}{L^\infty(\Omega)} \leq \frac{C}{\ep},
            \end{equation}
            where $C>0$ does not depend on $\ep$.
        \end{prop}
        \begin{proof}
        By expanding the first equation in \eqref{euler lagrange gl}, using in particular \eqref{coulomb gauge}, we get
        $$-\Delta \u = \frac{\u (a_\ep-|\u|^2)}{\ep^2} - 2i A\u \cdot \nabla \u - |A|^2 \u^2.$$
        Moreover, from the boundary conditions $\nabla_A \u\cdot \nu =0$ and $A\cdot \nu=0$ on $\partial\Omega$, we get
        $$\frac{\partial \u}{\partial \nu} = 0 \text{ on }\partial \Omega.$$
        Therefore, $\u$ satisfies \eqref{neumann condition grad estimate}. Combining this with \eqref{maxprinc|u|} and our bound on $\|A\|_{L^\infty(\Omega)}$, we deduce that 
        $$\norm{\nabla \u}{L^\infty(\Omega)}^2 \leq C\left(\frac{1}{\ep^2}+\|A\|_{L^\infty(\Omega)}\norm{\nabla \u}{L^\infty(\Omega)}+\|A\|_{L^\infty(\Omega)}^2\right)
        \leq \frac{C}{\ep}\left(\frac1\ep +\norm{\nabla \u}{L^\infty(\Omega)} \right),$$
        from where it follows that
        $$\norm{\nabla \u}{L^\infty(\Omega)} \leq \frac{C}{\ep}.$$
        Finally, using Proposition \ref{proprho_ep}, we get
        $$\|\nabla u\|_{L^\infty(\Omega)} = \left\|\nabla \left(\frac{\u}{\pmin}\right)\right\|_{L^\infty(\Omega)} \leq C\left(\|\nabla \u\|_{L^\infty(\Omega)} + \|\nabla \pmin\|_{L^\infty(\Omega)} \right) \leq \frac{C}{\ep}.$$
        \end{proof}
        The gradient bound \eqref{coulomb grad estimate} plays a crucial role in the next clearing out result.
        \begin{prop}\label{clearing out}
            Let $(u,A) \in H^1(\Omega,\C) \times H^1(\Omega,\R^2)$ be a configuration such that 
            $$
            \norm{\nabla |u|}{L^{\infty}(\Omega)} \leq \frac{C}{\varepsilon}\quad \mathrm{and} \quad \fen(u,A) = o(1).
            $$
            Then $\norm{1-|u|}{L^\infty(\Omega)} = o(1)$.
        \end{prop}

        \begin{proof}
            Observe that, since $\rho_\ep\geq \sqrt{b}$ (recall Proposition \ref{proprho_ep}), we have
            $$
            b^2F_\ep(u,A)\leq \fen(u,A)=o(1).
            $$
            Therefore, the proposition directly follows from the classical clearing out result for the homogeneous Ginzburg--Landau energy, which goes back to the seminal work of Bethuel, Brezis and Helein \cite{bethuel-brezis-helein}*{Theorem III.3}.

        \end{proof}
    \subsection{Vorticity estimate}
        Recall that the vorticity is defined as
        $$\mu(u,A) = \curl(\ip{iu}{\nabla_A u}+A).$$
        It is well known that in the homogeneous case, under adequate bounds on the free energy $F_\ep(u,A)$, $\mu(u,A)$ essentially acts as a sum of Dirac masses centered at the vortices when tested against sufficiently regular functions vanishing on the boundary. This also happens in the inhomogeneous case, since $b^2 F_\ep(u,A) \leq \fen(u,A)$. We have the following version of \cite{libro-ss}*{Theorem 6.1}.
        \begin{prop}\label{prop:vorticityestimate}
            Let $\mathcal{B} = \{B_i\}_i = \{B(a_i,r_i)\}_i$ be a finite collection of disjoint closed balls and $\varepsilon > 0$ such that 
            \begin{equation}\label{contención}\left\{x \in \Omega_\varepsilon \colon ||u(x)|-1| \geq \frac{1}{2}\right\} \subseteq \bigcup_{i} B_i,\end{equation}
            where $\Omega_\ep \colonequals \{x \in \Omega \colon \operatorname{dist}(x,\partial \Omega) > \ep\}$. Then, for any $r= \sum_i r_i \leq 1, \varepsilon \leq 1$, there exists a universal constant $C>0$ such that
            \begin{equation}\label{jacobian estimate}
                \norm{\mu - 2\pi \sum_i d_{B_i} \delta_{a_i}}{\left(C^{0,1}_0(\Omega)\right)^*} \leq C \max\{\varepsilon,r\} \left(1+\frac{M}{b^2}\right),
            \end{equation}
            where $d_{B_i} = \deg(u,\partial B_i)$ if $B_i \subset \Omega_\varepsilon$ and 0 otherwise, $M=\fen(u,A)$, and $\left(C^{0,1}_0(\Omega)\right)^*$ is the dual space of $C^{0,1}_0(\Omega) = W^{1,\infty}_0(\Omega)$.
        \end{prop} 
        \begin{proof}
             Since $b^2 F_\ep(u,A) \leq \fen(u,A)$, the proof is exactly the same as the proof of \cite{libro-ss}*{Theorem 6.1}.
        \end{proof}

    \subsection{Approximation of the Meissner state}
         We want to estimate the minimal energy among vortexless configurations. Heuristically, a good starting point would be to consider a pair of the form $(\sqrt{a_\ep},A)$, where $A$ minimizes $GL_{\varepsilon}(\sqrt{a_\varepsilon},\ \cdot \ )$ in a suitable space. However, $a_\varepsilon$ may not be in $H^1(\Omega)$. This leads to the introduction of the aforementioned function $\rho_\ep$. 
         
         Taking into account that $\rho_\ep$ is essentially a regularized version of $\sqrt{a_\ep}$, we proceed to the task of minimizing $GL_\ep(\rho_\ep, \ \cdot \ )$. For convenience, we next work with $h_{\ex} A$ instead of $A$.
         
         Observe that from the energy decoupling \eqref{LM eq}, we have
        $$GL_\ep(\pmin,h_{\ex} A) = E_\ep(\pmin) + \frac{1}{2}\int_\Omega \pmin^2 h_{\ex}^2 |A|^2 + h_\ex^2|\curl A - 1|^2.$$
        This leads us to look for vector fields $A$ that minimize the reduced energy functional
        $$J(A) = \frac{1}{2}\int_\Omega \pmin^2|A|^2 + |\curl A - 1|^2.$$
        Without loss of generality, we can look for minimizers in the Coulomb gauge, that is, vector-fields that satisfy \eqref{coulomb gauge}. Note that $J$ is strictly convex. Furthermore, since we look for minimizers in the Coulomb gauge, we have $\norm{A}{H^1(\Omega)} \leq C \norm{\curl A}{L^2(\Omega)}$. It follows that $J$ is strictly convex and coercive in the space
        $$\left\{A\in H^1(\Omega, \R^2) \ \colon \ \diver A=0\ \mbox{in }\Omega,\ A\cdot \nu=0 \ \mbox{on }\partial\Omega\right\}$$ and, as a result, there exists a unique minimizer $A^0_\ep$ of $J$ in this space. This minimizer satisfies the associated Euler--Lagrange equation
        \begin{equation}\label{pde A0}
        \left\{
        \begin{array}{rccl}
            -\nabla^\perp \curl A^0_\ep +\pmin^2 A^0_\ep&=&0&\mbox{in }\Omega\\
            \curl A^0_\ep&=&1&\mbox{on }\partial\Omega.
        \end{array}
        \right.
        \end{equation}
        This in particular means that (recall that $0<b\leq \rho_\ep^2\leq 1$ in $\Omega$)
        $$A^0_\ep = \frac{\nabla^\perp \curl A^0_\ep}{\pmin^2}\quad \mbox{in}\ \Omega.$$ 
        Also, letting $\hfield = \curl A^0_\ep$, by taking the $\curl$ of the PDE in \eqref{pde A0}, we deduce that $\hfield$ solves
        \begin{equation}\label{pde h0}
        \left\{
        \begin{array}{rccl}
            -\diver \left(\dfrac{\nabla \hfield}{\pmin^2}\right) +\hfield &=&0&\mbox{in }\Omega\\
            \hfield&=&1&\mbox{on }\partial\Omega.
        \end{array}
        \right.
        \end{equation}
        Finally, we let $\xfield = 1-\hfield$ to deduce that \eqref{pde xi} holds true. It is worth remarking that $\xfield$ is the analog of the function $\xi_0$ that appears in the analysis of the Ginzburg--Landau energy functional without pinning (i.e. \eqref{GLenergy} when $a_\ep\equiv 1$ in $\Omega$); see for instance \cite{sylvia-ccm}. 
     \begin{prop} We have that
        \begin{align}
            0 \leq \hfield \leq 1&\quad \mathrm{in}\ \overline\Omega,\label{h maximum principle}\\
            0 \leq \xfield \leq 1&\quad \mathrm{in}\ \overline\Omega.\label{xi maximum principle}
        \end{align}
    \end{prop}
    \begin{proof}
        From \eqref{pde h0}, by applying the maximum principle, we deduce that $\hfield \leq \max_{\partial \Omega} {\hfield}^+ = 1$ and $\hfield \geq -\max_{\partial \Omega} {\hfield}^- = 0$, where ${\hfield}^+ = \max\{\hfield,0\}$ and ${\hfield}^- = -\min\{\hfield,0\}$. The bounds for $\xfield$ follow immediately since $\xfield = 1-\hfield$.
    \end{proof}    
    An elemental consequence of the preceding proposition is that 
    \begin{equation}\label{xi h1 estimate}
        \norm{\xfield}{H^1(\Omega)} \leq C,
    \end{equation}
    where $C>0$ \emph{does not depend on $\ep$}. To see this, we test the equation in \eqref{pde xi} against $\xfield$ and use \eqref{xi maximum principle} and $\pmin^2 \leq 1$.
    $$\norm{\xfield}{H^1(\Omega)}^2 \leq \int_\Omega \frac{|\nabla \xfield|^2}{\pmin^2} + |\xfield|^2 = \int_\Omega \xfield \leq |\Omega|.$$
    Analogously, 
    \begin{equation}\label{boundH1hfield}\norm{\hfield}{H^1(\Omega)} \leq C.\end{equation}
    However, a rather surprising fact is that such a bound also holds for $W^{1,\infty}_0(\Omega)$.
    \begin{prop}\label{prop:linftyxi}
        We have that
        \begin{equation}\label{xi lipschitz estimate}
        \norm{\nabla \xfield}{L^\infty(\Omega)} \leq C,
        \end{equation}
        where $C>0$ does not depend on $\ep$.
    \end{prop}
    \begin{proof}
        In this proof, $C>0$ denotes a constant independent of $\ep$ that may change from line to line.
        Recall that $A^0_\ep = -\frac{\nabla^\perp \xfield}{\pmin^2}$. Therefore, using Proposition \ref{proprho_ep}, we have
        \begin{equation}\label{linf}\|\nabla \xfield\|_{L^\infty(\Omega)} = \|\pmin^2 A^0_\ep\|_{L^\infty(\Omega)} \leq \|A^0_\ep\|_{L^\infty(\Omega)}.\end{equation}
        On the other hand, since $A^0_\ep$ satisfies \eqref{coulomb gauge}, \eqref{coulomb h2 estimate} yields
        $$\norm{A^0_\ep}{H^2(\Omega)} \leq C\norm{\curl A^0_\ep}{H^1(\Omega)} = C \norm{\hfield}{H^1(\Omega)}.$$
        Combining with \eqref{boundH1hfield}, we deduce that
        $$\norm{A^0_\ep}{H^2(\Omega)} \leq C,$$
        which, by Sobolev embedding, yields
        $$\norm{A^0_\ep}{L^\infty(\Omega)} \leq C.$$
        Inserting this in \eqref{linf} concludes the proof.
        
    \end{proof}

        \begin{prop}\label{prop:liminf} We have that
            $$
            \liminf_{\ep \to 0} \max_\Omega \xi_\ep>0.
            $$            
        \end{prop}
        \begin{proof}
            Let us assume towards a contradiction that there exists a sequence $\{\ep_n\}_{n\in \N}$ such that
            $$
            \max_\Omega \xi_{\ep_n}\to 0\quad \mathrm{as}\ n\to \infty.
            $$
            By testing \eqref{pde xi} by $\xi_{\ep_n}$ and integrating by parts, we find
            $$
            \int_\Omega \frac{|\nabla \xi_{\ep_n}|^2}{\rho_{\ep_n}^2}+\int_{\Omega}\xi_{\ep_n}^2=\int_{\Omega}\xi_{\ep_n}.
            $$
            Since $\rho_{\ep_n}^2\leq 1$, we deduce that
            $$
            \|\xi_{\ep_n}\|^2_{H^1(\Omega)}\leq \int_{\Omega}\xi_{\ep_n}\leq |\Omega|\max_\Omega \xi_{\ep_n}.
            $$
            Thus, $\|\xi_{\ep_n}\|^2_{H^1(\Omega)}\to 0$ as $n\to \infty$. On the other hand, by testing \eqref{pde xi} by $v\in H_0^1(\Omega)$ and integrating by parts, we find 
            \begin{equation}\label{passlim}
            \int_\Omega \frac{\nabla \xi_{\ep_n}\cdot \nabla v}{\rho_{\ep_n}^2}+\int_{\Omega}\xi_{\ep_n}v=\int_{\Omega}v.
            \end{equation}
            Using $b\leq \rho_\ep^2$ and the Cauchy--Scharwz inequality, we find
            $$
            \left| \int_\Omega \frac{\nabla \xi_{\ep_n}\cdot \nabla v}{\rho_{\ep_n}^2}\right |\leq b^{-1} \left|\int_\Omega \nabla \xi_{\ep_n}\cdot \nabla v\right|\leq b^{-1}\|\nabla \xi_{\ep_n}\|_{L^2(\Omega)}\|\nabla v\|_{L^2(\Omega)}\to 0\quad \mathrm{as}\ n\to \infty.
            $$
            Similarly,
            $$
            \left|\int_\Omega \xi_{\ep_n}v \right|\leq \|\xi_{\ep_n}\|_{L^2(\Omega)}\|v\|_{L^2(\Omega)}\to 0\quad \mathrm{as}\ n\to \infty.
            $$
            Hence, passing to the limit $n\to\infty$ in \eqref{passlim}, we find $\int_\Omega v=0$ for any $v\in H_0^1(\Omega)$, which is a contradiction.
        \end{proof}
        \begin{remark}\label{remarkmaxpsi2}
            This result immediately yields $\liminf_{\ep\to 0} \psi_\ep>0$. Moreover, since $\xi_\ep=0$ on $\partial\Omega$, we have $\psi_\ep=0$ on $\partial\Omega$. We then deduce that there exists $d>0$, independently of $\ep$, such that $\mathrm{dist}\left(\mathrm{argmax}_\Omega(\psi_\ep),\partial\Omega\right)>d$ for any $\ep>0$.
            \end{remark}
        \subsection{Proof of Proposition \ref{prop-splitting}}
        We now are ready to provide a proof for our energy-splitting.       
        \begin{proof}
            From \eqref{LM eq}, we have 
            \begin{align}\label{eq0}
                GL_\varepsilon(\mathbf{u,A}) &= E_\varepsilon(\pmin u, \mathbf{A}) + \frac{1}{2} \int_\Omega |\operatorname{curl} \mathbf{A} - h_{\ex}|^2\\
                &= E_\varepsilon(\pmin) + \frac{1}{2}\int_\Omega \pmin^2|\nabla_\mathbf{A} u|^2 + \pmin^4 \frac{(1-|u|^2)^2}{2\varepsilon^2} + |\operatorname{curl} \mathbf{A} - h_{\ex}|^2\notag.
            \end{align}
            By expanding the square $|\nabla_\mathbf{A} u|^2$ and integrating by parts (recall from \eqref{pde xi} that $\xfield=0$ on $\partial \Omega$), we find
            \begin{align}\label{eq1}
                \int_\Omega \pmin^2 |\nabla_\mathbf{A} u|^2 &= \int_\Omega \pmin^2 \left|\nabla_{A} u + ih_{\ex}\frac{\nabla^\perp \xfield}{\pmin^2} u\right|^2 \\
                &= \int_\Omega \pmin^2\left(|\nabla_{A} u|^2 + h_{\ex}^2 \frac{|\nabla \xfield|^2}{\pmin^4}|u|^2 + 2\frac{h_{\ex}}{\pmin^2} \ip{\nabla_{A} u}{iu} \cdot \nabla^\perp \xfield \right) \notag\\
                &= \int_\Omega \pmin^2 |\nabla_{A} u|^2 + h_{\ex}^2 \frac{|\nabla \xfield|^2}{\pmin^2}|u|^2 - 2h_{\ex} \curl(\ip{iu}{\nabla_A u}) \xfield.\notag
            \end{align}       
            We now expand the square $|\operatorname{curl}\mathbf{A} - h_{\ex}|^2$, which yields
            \begin{align}\label{eq2}
                \int_\Omega |\operatorname{curl}\mathbf{A} - h_{\ex}|^2 &= \int_\Omega \left|\curl A + h_\ex \curl A_\ep^0- h_{\ex}\right|^2\\
                &= \int_\Omega \left|\curl A + h_\ex h_\ep^0 - h_{\ex}\right|^2\notag\\
                &=\int_\Omega |\curl A + h_\ex(1-\xfield) -h_\ex|^2\notag\\
                &= \int_\Omega |\curl A - h_\ex\xfield|^2\notag\\
                &= \int_\Omega |\operatorname{curl}A|^2 + h_{\ex}^2 |\xfield|^2 - 2h_{\ex} \xfield \operatorname{curl}A\notag. 
            \end{align}            
            Inserting \eqref{eq1} and \eqref{eq2} into \eqref{eq0}, we deduce that
            \begin{equation}\label{partial split}
                GL_\ep(\u,\A) = E_\ep(\pmin) + \fen(u,A) - h_\ex \int_\Omega \mu(u,A) \xfield + \frac{h_\ex^2}{2}\left( \frac{|\nabla \xfield|^2}{\pmin^2} |u|^2 + |\xfield|^2 \right).
            \end{equation}
            
            Let us now write $GL_\ep(\pmin, h_{\ex}A_\ep^0)$ in terms of the energies of $\pmin$ and $\xfield$. We have
            \begin{align}
                GL_\varepsilon(\pmin,h_{\ex} A^0_\ep) &= E_\ep(\pmin,h_\ex A^0_\ep) + \frac{h_{\ex}^2}{2}\int_\Omega |\curl A^0_\ep - 1|^2 \notag\\
                &\stackrel{\eqref{LM eq}}{=} E_\ep(\pmin) + \frac{h_\ex^2}{2} \int_\Omega \pmin^2 |A^0_\ep|^2 +  |h_\ep^0-1|^2 \notag\\
                &= E_\ep(\pmin) + \frac{h_\ex^2}{2} \int_\Omega \pmin^2 \frac{|\nabla \xfield|^2}{\pmin^4} + |\xfield|^2=E_\ep(\pmin) + \frac{h_\ex^2}{2}\int_\Omega \frac{|\nabla \xfield|^2}{\pmin^2} +|\xfield|^2.\label{meissner energy}
            \end{align}
            Therefore, by writing $|u|^2$ as $1 + (|u|^2 - 1)$, we have
            \begin{align*}
                \frac{h_\ex^2}{2} \int_\Omega \left( \frac{|\nabla \xfield|^2}{\pmin^2} |u|^2 + |\xfield|^2 \right) &= \frac{h_\ex^2}{2} \int_\Omega \left(\frac{|\nabla \xfield|^2}{\pmin^2} + |\xfield|^2\right) + \frac{h_\ex^2}{2}  \int_\Omega \frac{|\nabla \xfield|^2}{\pmin^2} (|u|^2-1) \\
                &\stackrel{\eqref{meissner energy}\&\eqref{defR0}}{=} GL_\ep(\pmin, h_\ex A_\ep^0) - E_\ep(\pmin) + R_0.
            \end{align*}
            By inserting this into \eqref{partial split}, we obtain \eqref{Split}.
        \end{proof}

        \begin{remark}
            Since $\pmin \geq \sqrt{b}$ and $\norm{\xfield}{H^1(\Omega)} \leq C$ for some $C>0$ independent of $\ep$ (recall \eqref{xi h1 estimate}), from \eqref{meissner energy} we deduce that
            \begin{equation}\label{meissner energy bound}
                GL_\ep(\pmin, h_\ex A_\ep^0) \leq E_\ep(\pmin) + Ch_{\ex}^2.
            \end{equation}
            On the other hand, by using the Cauchy--Schwarz inequality, \eqref{xi lipschitz estimate}, and $\pmin \geq \sqrt{b}$, we deduce that %
            \begin{equation}\label{boundR0}|R_0| \leq Ch_\ex^2 \norm{|u|^2-1}{L^2(\Omega)} \leq Ch_\ex^2 \ep \fen(u,A)^\frac{1}{2}.\end{equation}
            This in particular means that $R_0 = o(1)$ under adequate upper bounds on $h_\ex$ and $\fen(u,A)$.
        \end{remark} 
   
\section{First critical field}\label{first cf}
\subsection{Heuristic derivation of \texorpdfstring{$\critfield$}{the first critical field}}\label{heuristicHc1}
Since the Meissner configuration is a good approximation of the global minimizer among vortexless configurations (as we shall see in the next section), we expect the occurrence of vortices in global minimizers $(\u,\A)$ essentially when $GL_\ep(\u,\A)< GL_\ep(\pmin,h_\ex A_\ep^0)$. By our splitting result \eqref{Split}, we know that this is equivalent to finding values of $h_\ex$ such that
$$
\fen(u,A) - h_\ex \int_\Omega \mu(u,A) \xfield +R_0 < 0.
$$
Using the ball construction method given by Proposition \ref{ball lower bound} to estimate $\fen(u,A)$ and the vorticity estimate \eqref{jacobian estimate} to approximate $\mu(u,A)$ by a sum of Dirac masses, after neglecting lower order terms, we find that this is possible if
$$
h_{\ex} > \frac{|\log \ep|}{2 \max_\Omega \frac{\xfield}{\pmin^2}} = \frac{|\log \ep|}{2 \max_\Omega \psi_\ep}=H_{c_1}^\ep.
$$

\subsection{Proof of Theorem \ref{crit field lower bound}}
    \begin{proof}
    Note that all the results in this theorem are gauge-invariant. Therefore, we may assume without loss of generality that $(\u,\A)$ is in the Coulomb gauge, that is, $\A$ satisfies \eqref{coulomb gauge}. Also, in this proof, $C>0$ denotes a constant independent of $\ep$ that might change from line to line.
    \begin{enumerate}[label=\textsc{\bf Step \arabic*},leftmargin=0pt,labelsep=*,itemindent=*,itemsep=10pt,topsep=10pt]
    \item \emph{(Proving that $\fen(u,A) \leq Ch_{\ex}^2$ for some $C>0$ independent of $\ep$).} 
    Since $(\u,\A)$ is a global minimizer, we have $GL_\ep(\u,\A) \leq GL_\ep(\pmin, h_\ex A_\ep^0)$. By integrating by parts the third term in the RHS of \eqref{Split} (recall that $\xi_\ep=0$ on $\partial\Omega$) and inserting the previous inequality, we deduce that
    \begin{align*}
        \fen(u,A) &= GL_\ep(\u,\A) - E_\ep(\pmin) + h_\ex \int_\Omega \mu(u,A) \xfield -R_0\\
        &\leq GL_\ep(\pmin, h_\ex A^0_\ep) - E_\ep(\pmin) + h_\ex \int_\Omega (\ip{iu}{\nabla_A u} + A) \cdot \nabla^\perp \xfield +|R_0|.
    \end{align*}
    By inserting \eqref{meissner energy bound} and using the Cauchy-Schwarz inequality, we obtain
    \begin{multline}\label{ineq0}
    \fen(u,A)\leq  Ch_\ex^2 + h_\ex\norm{u}{L^2(\Omega)}\norm{\nabla_A u}{L^2(\Omega)} \norm{\nabla \xfield}{L^\infty(\Omega)}
        \\+h_\ex\norm{A}{L^2(\Omega)}\norm{\nabla \xfield}{L^2(\Omega)}+|R_0|.
    \end{multline}
    Since $(\u,\A)$ solves $\eqref{euler lagrange gl}$, we have \eqref{maxprinc|u|}. Combining this with $\norm{\nabla_A u}{L^2(\Omega)} \leq \fen(u,A)^{\frac{1}{2}}$ and \eqref{xi lipschitz estimate}, yields that
    \begin{equation}\label{ineq1}
    \norm{u}{L^2(\Omega)}\norm{\nabla_A u}{L^2(\Omega)} \norm{\nabla \xfield}{L^\infty(\Omega)}\leq C\fen(u,A)^{\frac{1}{2}}.
    \end{equation}
    Moreover, since both $\A$ and $A^0_\ep$ are in the Coulomb gauge, we deduce that $A$ also satisfies \eqref{coulomb gauge}. Hence, using \eqref{coulomb h1 estimate}, we get that
    $$
    \norm{A}{L^2(\Omega)} \leq \norm{A}{H^1(\Omega)} \leq C \norm{\curl A}{L^2(\Omega)} \leq C \fen(u,A)^{\frac{1}{2}},
    $$
    which combined with \eqref{xi h1 estimate} yields
    \begin{equation}\label{ineq2}
    \norm{A}{L^2(\Omega)} \norm{\nabla \xfield}{L^2(\Omega)} \leq C \fen(u,A)^{\frac{1}{2}}.
    \end{equation}
    Finally, by combining \eqref{ineq0} with \eqref{ineq1}, \eqref{ineq2}, and \eqref{boundR0}, we obtain
    \begin{align*}
        \fen(u,A) &\leq C\left(h_\ex^2 + h_\ex \fen(u,A)^\frac{1}{2}+h_\ex^2\ep \fen(u,A)^\frac12\right)\\
                  &\leq  C\left(h_\ex^2 + h_\ex \fen(u,A)^\frac{1}{2}\right).
    \end{align*}
    It follows that
    \begin{equation}\label{global minimizing energy bound}
        \fen(u,A) \leq Ch_\ex^2.
    \end{equation}
    
    \item \emph{(Estimates for $\fen(u,A)$ and $\norm{\mu(u,A)}{\left(C^{0,1}_0\right)^*}$. Proof of item (2)).}
    From \eqref{global minimizing energy bound} and $h_\ex = O(|\log \ep|)$, we have
    \begin{equation}\label{global minimizing energy bound2}
        \fen(u,A) \leq C|\log \ep|^2.
    \end{equation}
    We can therefore apply Proposition \ref{ball lower bound}, to obtain a finite collection of disjoint balls $\{B_i\}_i=\{B(a_i,r_i)\}_i$ with $\sum_i r_i \leq r= |\log \varepsilon|^{-\beta}$, where $\beta>0$ will be chosen later, containing $\left\{||u|-1| \geq \frac{1}{2} \right\}$ such that
\begin{align}
    \fen(u,A) &\geq \pi  \sum_{i} \pmin^2( \underline{a_i})|d_{B_i}| \left( \log \frac{|\log \varepsilon|^{-\beta}}{\tilde D \varepsilon} - C \right)\nonumber \\
    &= \pi  \sum_{i} \pmin^2( \underline{a_i})|d_{B_i}| (|\log \varepsilon| - \beta\log|\log\varepsilon| - \log \tilde D  - C) \nonumber\\
    &\stackrel{\eqref{degree bound}}{\geq}  \pi  \sum_{i} \pmin^2( \underline{a_i})|d_{B_i}| (|\log \varepsilon| - \beta\log|\log\varepsilon| - C\log \frac{\fen(u,A)}{|\log \ep|}  - C) \nonumber\\
    &\stackrel{\eqref{global minimizing energy bound2}}{\geq} \pi  \sum_{i} \pmin^2( \underline{a_i})|d_{B_i}| (|\log \varepsilon| - C\log|\log\varepsilon|)\label{free energy bound fcf},
\end{align}
where $\underline{a_i} \in B_i$ is such that $\pmin^2(\underline{a_i}) = \min_{B_i} \pmin^2$.

On the other hand, applying Proposition \ref{prop:vorticityestimate}, we have
\begin{align*}
    \left|h_{\ex}\int_\Omega \mu(u,A) \xfield \right|&\stackrel{\eqref{jacobian estimate}}{\leq} 2\pi h_{\ex} \sum_i |d_i| \xi_\varepsilon(a_i) + C h_\ex r(1+\fen(u,A))\norm{\nabla \xfield}{L^{\infty}(\Omega)} \\
    &\stackrel{\eqref{xi lipschitz estimate} \& \eqref{global minimizing energy bound2}}{\leq} 2\pi h_{\ex} \sum_i |d_i| \xi_\varepsilon(a_i) + O(|\log \varepsilon|^{3-\beta})
\end{align*}
It also follows from \eqref{xi lipschitz estimate} that
$$|\xfield(a_i) - \xfield(\underline{a_i})| \leq \norm{\nabla \xfield}{L^\infty(\Omega)}|a_i - \underline{a_i}| \leq Cr_i \leq C|\log \ep|^{-\beta} .$$
Therefore, we have
\begin{align*}
    \left| h_{\ex}\int_\Omega \mu(u,A) \xfield \right|  &\leq 2\pi h_{\ex} \sum_{i}|d_i|\xfield(\underline{a_i}) + C|\log \ep|^{-\beta} h_\ex \sum_i |d_i| + O(|\log \varepsilon|^{3-\beta})\\
    &\stackrel{\eqref{degree bound}}{\leq} 2\pi h_{\ex} \sum_{i}|d_i|\xfield(\underline{a_i}) + C |\log \ep|^{-\beta} h_\ex \frac{\fen(u,A)}{|\log \ep|} + O(|\log \varepsilon|^{3-\beta})\\
    &\stackrel{\eqref{global minimizing energy bound2}}{\leq} 2\pi h_{\ex} \sum_{i}|d_i|\xfield(\underline{a_i})+ O(|\log \varepsilon|^{3-\beta}).
\end{align*}
Thus, by choosing $\beta > 3$, we get
\begin{equation}\label{vorticity convergence}
   \left| h_{\ex} \int_\Omega \mu(u,A) \xfield \right| \leq 2\pi h_{\ex} \sum_i |d_i| \xfield(a_i) + o(1).
\end{equation}
Combining \eqref{free energy bound fcf} and \eqref{vorticity convergence}, we deduce that
\begin{multline}\label{f-mu lower bound}
    \fen(u,A)-h_{\ex}\int_\Omega \mu(u,A) \xfield \geq \\ \pi \sum_{i}\pmin^2(\underline{a_i})|d_i|\left(|\log \varepsilon| - C \log|\log \varepsilon|
    -2h_{\ex}\psi_\varepsilon(\underline{a_i})\right)+o(1).
\end{multline}
Therefore, since $h_\ex \leq \critfield - K_0 \log |\log \ep|$, we have
\begin{align*}
    |\log \ep| - C \log |\log \ep| - 2h_\ex \psi_\ep(\underline{a_i})
    &\geq |\log \ep| - C \log |\log \ep| - 2h_\ex \max_\Omega \psi_\ep\\
    &\geq \log |\log \ep|\left(2\max_\Omega \psi_\ep K_0 - C\right).
    \end{align*}
Remark \ref{remarkmaxpsi} (or Proposition \ref{prop:liminf}) then allows us to choose $K_0>0$, independently of $\ep$, so that
$$
2\max_\Omega\psi_\ep K_0-C=1.
$$
Inserting this into \eqref{f-mu lower bound}, we find
\begin{equation}\label{f-mu lower bound2}
    \fen(u,A)-h_{\ex}\int_\Omega \mu(u,A) \xfield \geq  \pi \sum_{i}\pmin^2(\underline{a_i})|d_i|\log |\log\ep| +o(1).
\end{equation}

Moreover, since $GL_{\varepsilon}(\mathbf{u,A}) \leq GL_\ep(\pmin, h_\ex A^0_\ep)$, it follows from \eqref{Split} that
$$\fen(u,A) - h_{\ex}\int_\Omega \mu(u,A) \xfield + R_0 \leq 0.$$
In addition,
\begin{equation}\label{o1 r0}
    |R_0| \leq C h_{\ex}^2 \varepsilon \fen(u,A)^{\frac{1}{2}} \stackrel{\eqref{global minimizing energy bound2}}{\leq} C\ep |\log \ep|^3 = o(1).
\end{equation}
Hence,
\begin{equation}\label{free energy vorticity bound}
    \fen(u,A) - h_{\ex}\int_\Omega \mu(u,A) \xfield \leq o(1).
\end{equation}
By combining \eqref{f-mu lower bound2} and \eqref{free energy vorticity bound}, using also $\rho_\ep^2\geq b$, we deduce that $\sum_i |d_i|=0$ and thus $d_i=0$ for all $i$. In turn, from \eqref{jacobian estimate} it follows that
\begin{equation}\label{vortexless vorticity}
    h_\ex\norm{\mu(u,A)}{\left(C^{0,1}_0(\Omega)\right)^*} \leq Ch_\ex r(1+\fen(u,A)) \leq C|\log \ep|^{3-\beta} = o(1).
\end{equation}
Therefore, item (2) is satisfied.

\item \emph{(Clearing out. Proof of items (1) and (3)).}
Since $(\u,\A)$ is in the Coulomb gauge, we have 
\begin{align*}
    \norm{A}{L^\infty(\Omega)} &\leq C(E_\ep(\pmin) + \fen(u,A))^\frac{1}{2}\\
    &\stackrel{\eqref{global minimizing energy bound2}}{\leq} C\left( \frac{1}{\ep^2} + |\log \ep|^2 \right)^{\frac{1}{2}}\\
    &\leq \frac{C}{\ep}.
\end{align*}
Then, it follows from \eqref{coulomb grad estimate} that
\begin{align*}
    \norm{\nabla |u|}{L^\infty(\Omega)} \leq \norm{\nabla u}{L^\infty(\Omega)} \leq \frac{C}{\ep}.
\end{align*}
On the other hand, by combining \eqref{free energy vorticity bound} with \eqref{vortexless vorticity}, we find
\begin{equation}\label{o1 vorticity}
    \fen(u,A) \leq h_\ex \int_\Omega \mu(u,A) \xfield +o(1)\stackrel{\eqref{xi lipschitz estimate}\&\eqref{vortexless vorticity}}{=} o(1).
\end{equation}
Hence, Proposition \ref{clearing out} yields that item (1) holds. 

Finally, we have
\begin{align*}
    GL_\ep(\u,\A) &= GL_\ep(\pmin, h_\ex A^0_\ep) + \fen(u,A) - h_\ex \int_\Omega \mu(u,A)\xfield + R_0\\
    &\stackrel{ \eqref{o1 r0} \&\eqref{vortexless vorticity}\& \eqref{o1 vorticity} }{=} GL_\ep(\pmin, h_\ex A^0_\ep) + o(1)
\end{align*}
This finishes the proof of item (3). 
\end{enumerate}
\end{proof} 

\subsection{Proof of Theorem \ref{crit field upper bound}}

\begin{proof}
    In this proof, we will construct a configuration of the form $(\mathbf{u,A}) = (\pmin u, 0+ h_{\ex} A^0_\ep)$, with a vortex of degree $1$ centered at $x^0_\ep \in \Omega$, where $x^0_\ep$ is such that 
    \begin{equation}\label{x0 is max}
        \psi_\ep(x^0_\ep) = \max_\Omega \psi_\ep.
    \end{equation}We will prove that the energy of such a configuration is much lower than the energy of the Meissner configuration, which in turn guaranties that global minimizers of \eqref{GLenergy} in this regime have vortices.
    \begin{enumerate}[label=\textsc{\bf Step \arabic*},leftmargin=0pt,labelsep=*,itemindent=*,itemsep=10pt,topsep=10pt]
    \item \emph{(Constructing the configuration).} Let $\Phi$ be a multiple of the fundamental solution of the Laplace's equation centered at $x_\ep^0$, that is,
    $$\Phi(x) = \log \frac{1}{|x-x^0_\ep|}.$$
    We begin by constructing a phase $\varphi$ in $\Omega \setminus \{x^0_\ep\}$ as follows. Let $\Theta$ be the phase of
    $$\frac{z-x^0_\ep}{|z-x^0_\ep|}.$$ 
    Since 
    $$
    -\Delta \Phi = 2\pi \delta_{x^0_\ep} = \curl \nabla \Theta\quad \mathrm{in}\ \Omega,
    $$
    we have that, in the sense of distributions,
    $$
    \curl(-\nabla^\perp \Phi - \nabla \Theta)= 0 \quad \mathrm{in}\ \Omega.
    $$
    Therefore, there exists $g$ such that $\nabla g = -\nabla^\perp \Phi - \nabla \Theta$. We let $\varphi = \Theta + g$. Observe that $\varphi$ is well defined modulo $2\pi$ in $\Omega \setminus \{x^0_\ep\}$ and satisfies the following relation
    \begin{equation}\label{phase grad equality}
        \nabla \varphi = -\nabla^\perp \Phi.
    \end{equation}

    Let $r_\ep = |\log \ep|^{-M}$, where $M > 0$ will be chosen later on, and consider the ball $B_\ep = B(x^0_\ep, r_\ep) \subset \Omega$. Notice that this condition holds for any $\ep$ sufficiently small in view of Remark \ref{remarkmaxpsi2}.
    
    We can now define $u$. For $x\in \Omega \setminus B_\ep$, we let $u(x) = e^{i\varphi(x)}$ and, for $x \in B_\ep$, we define
    $$u(x) = \frac{1}{f(R_\ep)} f\left( \frac{|x-x^0_\ep|}{\varepsilon} \right)e^{i\varphi(x)},$$
    where $R_\ep$ is such that $r_\ep=\ep R_\ep$ and $f \colon \R_+ \to \R_+$ is a function such that $f(0)=0$, $f(r) \to 1$ as $r \to \infty$ and satisfies the following asymptotic estimate 
    \begin{equation}\label{R-asympt}
        \frac{1}{2} \int_0^{R} \left(f'(r)^2 + \frac{f(r)^2}{r^2} + \frac{(1-f(r)^2)^2}{2} \right) 2\pi rdr = \pi \log R + O(1)\quad \mathrm{as}\ R\to\infty.
    \end{equation}
    The function $f$ is the modulus of what is referred to as \emph{the degree-one radial solution} \cite{libro-ss}*{Definition 3.6}, and its existence and properties are given by \cite{libro-ss}*{Proposition 3.11}.

    \item \emph{(Estimating the energy inside $B_\ep$).}
    Let $k_{\varepsilon} = \sup_{x \in B_\ep} |\pmin^2(x)-\pmin^2(x_\ep^0)|$. Using that $\rho_\ep^2\leq 1$, We have
    \begin{align*}
        F_{\varepsilon, \rho_\varepsilon, B_\ep}(u,0)=\frac{1}{2}\int_{B_\ep} \pmin^2|\nabla u|^2 + \pmin^4 \frac{(1-|u|^2)^2}{2\varepsilon^2} &\leq \frac{1}{2}\int_{B_\ep} \pmin^2\left(|\nabla u|^2 + \frac{(1-|u|^2)^2}{2\varepsilon^2} \right) \\
        & \leq \frac12(\pmin^2(x^0_\ep) + k_\ep) \int_{B_\ep} |\nabla u|^2 + \frac{(1-|u|^2)^2}{2\varepsilon^2}.
    \end{align*}
    We now estimate the integral that appears in the RHS of the last inequality. Since $|\nabla u|^2 = |\nabla |u||^2 + |u|^2|\nabla \varphi|^2$, it follows by letting $r = \frac{|x-x_\ep^0|}{\varepsilon}$ and performing a direct calculation that
    \begin{align*}
        \frac{1}{2}\int_{B_\ep} |\nabla u|^2 + \frac{(1-|u|^2)^2}{2\varepsilon^2} = \frac{1}{2} \int_{B_\ep}\left( \frac{f'(r)^2}{\varepsilon^2 f(R)^2} + \frac{f(r)^2}{f(R)^2}|\nabla \Phi(x)|^2 + \frac{1}{2\varepsilon^2}\left(1-\frac{f(r)^2}{f(R)^2} \right)^2 \right) dx.
    \end{align*}
    Note that $|\nabla \Phi(x)| = \frac{1}{|x-x^0_\ep|} = \frac{1}{\varepsilon r}$. By changing the variable of integration to $r$, we obtain
    \begin{align*}
        \frac{1}{2}\int_{B_\ep} |\nabla u|^2 + \frac{(1-|u|^2)^2}{2\varepsilon^2} = \frac{1}{2} \int_0^{R_\ep} \left(\frac{f'(r)^2}{f(R)^2} + \frac{f(r)^2}{f(R)^2} \frac{1}{r^2} + \frac{1}{2}\left(1-\frac{f(r)^2}{f(R)^2} \right)^2 \right) 2\pi r dr.
    \end{align*}
    Since $R_\ep \to \infty$ as $\ep \to 0$, we have $f(R_\ep) \to 1$ as $\ep\to 0$. Therefore, it follows from \eqref{R-asympt} that
    \begin{align*}
        \frac{1}{2}\int_{B_\ep} \pmin^2|\nabla u|^2 + \pmin^4 \frac{(1-|u|^2)^2}{2\varepsilon^2} &\leq (\pmin^2(x^0_\ep)+k_\ep)(\pi \log R_\ep + O(1))\\
        &=(\pmin^2(x^0_\ep)+k_\ep)(\pi \log r_\ep - \pi \log \ep + O(1))\\
        &=(\pmin^2(x^0_\ep)+k_\ep)(\pi |\log \ep| - \pi M \log |\log \ep| ).
    \end{align*}
    From the hypothesis on $\pmin$ we have $k_\ep \leq [\pmin^2]_{C^{0,\alpha}(\Omega)} r_\ep^\alpha \leq |\log \ep|^{m-\alpha M}$. Therefore, by choosing a sufficiently large $M$, we have
    \begin{equation}\label{energy inside B}
        \frac{1}{2}\int_{B_\ep} \pmin^2|\nabla u|^2 + \pmin^4 \frac{(1-|u|^2)^2}{2\varepsilon} \leq \pmin^2(x^0_\ep)\left(\pi |\log \ep| - \pi M \log |\log \ep|\right).
    \end{equation}

    \item \emph{(Estimating the energy outside $B_\ep$).} Let $C(\Omega)=\mathrm{diam}(\Omega)$. Since $|u|=1$ outside $B_\ep$, we have $\nabla|u|=0$ and thus
    \begin{align*}
        \int_{\Omega \setminus B_\ep} |\nabla u|^2 &= \int_{\Omega \setminus B_\ep}|\nabla |u||^2 + |u|^2|\nabla \varphi|^2\\
        &\stackrel{\eqref{phase grad equality}}{=} \int_{\Omega \setminus B_\ep} |\nabla \Phi|^2.
    \end{align*}
    Therefore, using once again that $\rho_\ep^2\leq 1$, we have
    \begin{align*}
        F_{\varepsilon, \rho_\varepsilon, \Omega\setminus B_\ep}(u,0)=\frac{1}{2}\int_{\Omega \setminus B_\ep} \pmin^2|\nabla u|^2 + \pmin^4 \frac{(1-|u|^2)^2}{2\varepsilon^2} &\leq \frac{1}{2}\int_{\Omega \setminus B_\ep} |\nabla \Phi|^2 \\
        &= \frac{1}{2}\int_{\Omega \setminus B_\ep} \frac{1}{|x-x^0_\ep|^2} dx\\
        &\leq \frac{1}{2}\int_{r_\varepsilon}^{C(\Omega)}\frac{1}{r^2} 2\pi r dr\\
        &= -\pi \log r_\varepsilon + O(1) \\
        &= \pi M \log |\log \varepsilon| + O(1).
    \end{align*}
    
    Hence, by combining the estimates obtained in {\bf Step 2} and {\bf Step 3}, we obtain the following upper bound for the free energy
    \begin{equation}\label{free energy upper bound}
        \fen(u,0) \leq \pi\left(\pmin^2(x^0_\ep)|\log \ep| + (1-\pmin^2(x^0_\ep))M \log|\log \ep|\right) + O(1).
    \end{equation}

    \item \emph{(Computation of the full Ginzburg--Landau energy of the constructed configuration).} Consider the configuration $(\mathbf{u,A}) = (\pmin u,0 -h_{\ex}A^0_\ep)$. We split $GL_\varepsilon(\mathbf{u,A})$ using \eqref{Split}, to obtain
    \begin{align}\label{splitupperbound}
        GL_\varepsilon(\mathbf{u,A}) - GL_\ep(\pmin, h_\ex A^0_\ep) &=  \fen(u,0) -h_{\ex}\int_\Omega \mu(u,0) \xfield + R_0 \notag\\
        &\stackrel{\eqref{free energy upper bound}}{\leq} \pi \pmin^2(x^0_\ep) |\log \varepsilon| + \left(1-\pmin^2(x^0_\ep)\right)\pi M \log|\log \varepsilon| \notag\\
        &\hspace*{5cm}- h_{\ex}\int_\Omega \mu(u,0) \xfield + R_0.
    \end{align}
    Here is where the hypothesis on $h_\ex$
    \begin{equation}\label{hex upper bound}
        \critfield + K^0 \log |\log \ep| \leq h_\ex \leq |\log \ep|^N
    \end{equation} plays its role. First, we have 
    \begin{equation}\label{boundR0upper}
    |R_0| \stackrel{\eqref{boundR0}}{\leq} Ch_\ex^2 \ep \fen(u,0)^\frac{1}{2} \stackrel{\eqref{free energy upper bound} \& \eqref{hex upper bound}}{\leq} C \ep |\log \ep|^{\frac12+2N} = o(1).
    \end{equation}
    On the other hand, since $|u|=1$ in $\Omega \setminus B_\ep$, from \eqref{jacobian estimate} it follows that (recall $r_\ep=|\log\ep|^M)$
    \begin{align*}
        \int_\Omega \mu(u,0) \xfield &\geq 2\pi \xfield(x^0_\ep) - Cr_\varepsilon \fen(u,0)\norm{\nabla \xfield}{L^\infty(\Omega)} \\
        &\stackrel{\eqref{free energy upper bound}}{\geq} 2\pi \xfield(x^0_\ep) -C|\log \varepsilon|^{-M} |\log \varepsilon|.
    \end{align*}
    Therefore, we have 
    $$h_\ex \int_\Omega \mu(u,0) \xfield \geq 2\pi h_\ex \xfield(x^0_\ep) -C|\log \ep|^{N-M+1} .$$
    By choosing a larger $M$ if necessary, we get
    \begin{equation}\label{hex mu lower bound}
        h_\ex \int_\Omega \mu(u,0) \xfield \geq 2\pi h_\ex \xfield(x^0_\ep) + o(1).
    \end{equation}
    Finally, by combining \eqref{splitupperbound}, \eqref{boundR0upper}, \eqref{hex mu lower bound}, and $-\rho_\ep^2(x_\ep^0)\leq -b$, we are led to
    \begin{align*}
        GL_\ep(\u,\A)&-GL_\ep(\pmin, h_\ex A^0_\ep) \\
        &\leq \pi\left(\pmin^2(x^0_\ep) |\log \ep| +(1-\pmin^2(x^0_\ep)) M \log|\log \ep| -2h_\ex \xfield(x_\ep^0) \right)+o(1)
        \\&\leq \pi |\log \ep|\left(\pmin^2(x^0_\ep) - \frac{\xfield(x^0_\ep)}{\max_\Omega \frac{\xfield}{\pmin^2}}\right)+ \pi \log |\log \ep|\left((1-b)M - 2K^0\xfield(x^0_\ep) \right)+o(1).
    \end{align*}
    Since $\psi_\ep = \frac{\xfield}{\pmin^2}$ achieves its maximum at $x^0_\ep$, the term of order $|\log \ep|$ in the RHS of the last inequality is equal to 0. Therefore, since $\liminf_{\ep \to 0} \max_\Omega \xfield > 0$ (see Remark \ref{remarkmaxpsi} or Proposition \ref{prop:liminf}), we may choose $K_0$, independently of $\ep$, such that we have
    \begin{equation}\label{vortex energy difference}
        GL_\varepsilon(\mathbf{u,A}) - GL_\ep(\pmin, h_\ex A^0_\ep) < -\log |\log \ep|.
    \end{equation}

   \item \emph{(Conclusion).} 
   Let $\left(\u_0, \A_0\right) = (\pmin u_0, A_0 + h_\ex A^0_\ep)$ be a vortexless configuration, that is, $|u_0| > c$ for some $c>0$ independent of $\ep$, such that $GL_\ep(\u_0, \A_0) \leq GL_\ep(\pmin, h_\ex A^0_\ep)$. We split its Ginzburg--Landau energy with \eqref{Split} to obtain
   $$0 > GL_{\ep}(\u_0,\A_0) - GL_\ep(\pmin,h_\ex A^0_\ep) = \fen(u_0,A_0) -h_\ex \int_\Omega \mu(u_0,A_0) \xfield + R_0.$$
   By integration by parts, we have (recall $\xi_\ep=0$ on $\partial \Omega$)
   $$\int_\Omega \mu(u_0, A_0) \xfield = \int_\Omega (\ip{iu_0}{\nabla_{A_0} u_0} + A_0) \cdot \nabla^\perp \xfield.$$
   Since $|u_0|>c$, we can write $u_0 = |u_0|e^{i \varphi_0}$. A direct calculation shows that
   $$\ip{iu_0}{\nabla_{A_0} u_0} + A_0 = (1-|u_0|^2)(\nabla \varphi_0- A_0) + \nabla \varphi_0.$$
   Integration by parts then yields
   $$\int_\Omega \nabla \varphi_0 \cdot \nabla^\perp \xfield = - \int_\Omega \xfield \curl \nabla \varphi_0 = 0.$$
   Hence, from the Cauchy--Schwarz inequality it follows that
   \begin{align*}
       h_\ex\left|  \int_\Omega \mu(u_0, A_0) \xfield \right|&= h_\ex \left|\int_\Omega (1-|u_0|^2)(\nabla \varphi_0 - A_0) \right|\\
       &\leq Ch_\ex \ep \norm{1-|u_0|^2}{L^2(\Omega)} \norm{\nabla \varphi_0 - A_0}{L^2(\Omega)}\\
       &\stackrel{\eqref{hex upper bound}}{\leq} C |\log \ep|^N \ep \fen(u_0, A_0)=o(1)\fen(u_0, A_0).
   \end{align*}
   On the other hand,
    \begin{equation*}
    |R_0| \stackrel{\eqref{boundR0}}{\leq} Ch_\ex^2 \ep \fen(u_0,A_0)^\frac{1}{2} \stackrel{\eqref{hex upper bound}}{\leq} C \ep |\log \ep|^{2N} \fen(u_0,A_0)^\frac{1}{2}.
    \end{equation*}
   Therefore, we have
   \begin{align*}
       0&>GL_{\ep}(\u_0,\A_0) - GL_\ep(\pmin,h_\ex A^0_\ep) \\
       &> \fen(u_0, A_0)(1-o(1)) - \ep|\log \ep|^{2N} \fen(u_0, A_0)^{\frac{1}{2}}.
    \end{align*}
    This implies that $\fen(u_0, A_0)^\frac12\leq C\ep |\log\ep|^{2N}=o(1)$ and therefore
    $$
    GL_\ep(\u_0, \A_0) - GL_\ep(\pmin, h_\ex A^0_\ep) = o(1).
    $$
    This means that $GL_\ep(\u, \A) \ll GL_\ep(\u_0, \A_0)$ for every vortexless configuration $(\u_0, \A_0)$. Hence, global minimizers of \eqref{GLenergy} in the regime \eqref{hex upper bound} do have vortices. This concludes the proof of the theorem.
    \end{enumerate}
\end{proof}

\begin{remark}
Notice that the hypothesis on $\pmin$ only plays a role at the end of {\bf Step 2}. Moreover, we can replace $ [\pmin^2]_{C^{0,\alpha}(\Omega)}\leq |\log\ep|^m$ by $[\pmin^2]_{C^{0,\alpha}(B_\ep)}\leq |\log \ep|^m$, that is, we only need a control over the H\"older seminorm around the points where the function $\psi$ achieves its maximum in $\overline \Omega$. 
\end{remark}

\section{Existence and uniqueness of a Meissner configuration above the first critical field}\label{above fcf}

\subsection{Proof of Theorem \ref{existence-above-fcf}} 
\begin{proof}
    In this proof, we use $C$ to denote a positive constant independent of $\ep$ that might change in each line.
    \begin{enumerate}[label=\textsc{\bf Step \arabic*},leftmargin=0pt,labelsep=*,itemindent=*,itemsep=10pt,topsep=10pt]
    \item \emph{(Construction of the locally minimizing vortexless configuration. Proof of items (1), (2) and (3)).} Fix $\beta \in (0,2-4\alpha)$ and let $$U = \left\{(\mathbf{u,A}) \in H^1(\Omega,\C) \times H^1(\Omega, \R^2) \colon \diver A=0\ \mbox{in }\Omega,\ A\cdot \nu=0 \ \mbox{on }\partial\Omega,\ \fen(u,A) < \varepsilon^\beta \right\}.$$ 
    First, let us prove that there exists a configuration $(\u_\ep, \A_\ep)$ that minimizes $GL_\ep$ over $\overline{U}$. Note that if $(\mathbf{u,A}) = \left(\pmin, h_{\ex}A^0_\ep \right)$, then $(u,A) = (1,0)$. This means that $\fen(1,0)=0$ and $A=0$ (trivially) satisfies $\eqref{coulomb gauge}$. It follows that $\left(\pmin, h_{\ex}A^0_\ep \right) \in U$ and therefore, $U \neq \emptyset$. 
    
    On the one hand, using Sobolev embedding and the Cauchy--Schwarz inequality, we find that each $(\u,\A) = (\pmin u, A + h_\ex A^0_\ep) \in U$ satisfies
    $$\norm{A}{H^1(\Omega)}^2 \stackrel{\eqref{coulomb h1 estimate}}{\leq} C \norm{\curl A}{H^1(\Omega)}^2 \leq C \fen(u,A)< C \ep^\beta,$$
    $$\norm{u}{L^4(\Omega)}^2 =\norm{u^2}{L^2(\Omega)}\leq C+\|1-|u|^2\|_{L^2(\Omega)}\leq C(1 + \ep\fen(u,A)^\frac12)< C(1 +\ep^{1+\frac\beta2}),$$
    $$\norm{\nabla u}{L^2(\Omega)} \leq C\norm{\nabla_A u}{L^2(\Omega)} + \norm{Au}{L^2(\Omega)} \leq \fen(u,A)^{\frac{1}{2}} + \norm{A}{L^4(\Omega)}\norm{u}{L^4(\Omega)}\leq C \ep^{\frac{\beta}{2}} .$$
    Hence, $U$ is bounded. 
    
    On the other hand, by writing
    \begin{align*}
        GL_\ep(\u, \A) = \frac{1}{2}\int_\Omega (|\nabla \u|^2 + |\A|^2|\u|^2 -2\ip{\nabla \u}{i\A \u}) + |\curl \A - h_\ex|^2 + \frac{(a_\ep - |\u|^2)^2}{2\ep^2},
    \end{align*}
    we deduce that $GL_\ep$ is $H^1$-weakly lower semicontinuous, since:
    \begin{itemize}
        \item The term $\int_\Omega |\nabla \u|^2 + |\curl \A - h_\ex|^2$ is convex and $H^1$-strongly continuous. Therefore, it is $H^1$-weakly lower semicontinuous.
        \item The term $\frac{1}{2}\int_\Omega \frac{(a_\ep-|\u|^2)^2}{2\ep^2}$ is $L^4$-strongly continuous and, by the Rellich--Kondrachov theorem, also $H^1$-weakly continuous.
        \item By the Cauchy--Schwarz inequality, the term $\int_\Omega |\A|^2|\u|^2$ is $L^4$-strongly continuous and, once again, by the Rellich-Kondrachov theorem it is also $H^1$-weakly continuous.
        \item The term $\int_\Omega \ip{\nabla \u}{i\A\u}$ is also $H^1$-weakly continuous. To see this, if $(\u_n, \A_n)$ weakly converges to $(\u,\A)$ in $H^1(\Omega, \C) \times H^1(\Omega, \R^2)$, then, by the Rellich-Kondrachov theorem, $(\u_n, \A_n)$ strongly converges to $(\u,\A)$ in $L^4(\Omega, \C) \times L^4(\Omega, \R^2)$ and therefore, $\A_n \u_n$ strongly converges to $\A \u$ in $L^2(\Omega)$. This means that $\int_\Omega \ip{\nabla \u_n}{i\A_n\u_n}$ converges to $\int_\Omega \ip{\nabla \u}{i\A\u}$.
    \end{itemize}
    Since $GL_\ep$ is $H^1$-weakly lower semicontinuous in a nonempty bounded set $U$, it follows that there exists $(\u_\ep, \A_\ep)$ that minimizes $GL_\ep$ over $\overline{U}$. Moreover, we have
      \begin{equation}\label{localmin energy bound}
         \fen(u_\ep, A_\ep) \leq \ep^\beta.
     \end{equation}
     We claim that $(\mathbf{u_\varepsilon,A_\varepsilon}) \in U$, which in turn implies that $(\u_\ep, \A_\ep)$ is a critical point and thus, a solution of $\eqref{euler lagrange gl}$. From now on we drop the $\varepsilon$ subscript.
    
    Since $(\mathbf{u,A})$ is a minimizing configuration in $\overline U$,
    \begin{equation}\label{up}GL_\varepsilon(\mathbf{u,A}) \leq  GL_\ep(\pmin, h_\ex A^0_\ep).
    \end{equation}    
    By combining \eqref{Split} with \eqref{up}, we deduce that
    \begin{equation}\label{boundforenergy}\fen(u,A) \leq h_\ex\int_\Omega \mu(u,A)\xi_\varepsilon - R_0.
    \end{equation}
    First, let us bound the vorticity term. Since $\fen(u,A) \leq \ep^\beta$, we can apply Proposition \ref{ball lower bound}, which provide us with a collection of balls $\mathcal{B} = \{B_i\}_i = \{B(a_i, r_i)\}$, with $\sum_i r_i\leq r=\varepsilon^{\mu}$ and where $\mu\in(\alpha,1)$ is a fixed number.
    
    By combining \eqref{localmin energy bound} and \eqref{degree bound}, we obtain
    $$
    \sum_i |d_{B_i}| \leq C\frac{\fen(u,A)}{|\log \ep|}\leq C\frac{\ep^{\beta}}{|\log\ep|}=o(1).
    $$   
    It follows that $\sum_i |d_{B_i}| = 0$, which implies $d_{B_i} = 0$ for all $i$.    
    Hence, it follows from \eqref{jacobian estimate} and the hypothesis $h_\ex \leq \ep^{-\alpha}$, that
    \begin{align}
    h_{\ex} \left| \int_\Omega \mu(u,A)\xi_\varepsilon \right| &\leq C \varepsilon^{-\alpha} \varepsilon^{\mu} \fen (u,A)\norm{\nabla \xfield}{L^{\infty}(\Omega)}\nonumber \\
    &\stackrel{\eqref{xi lipschitz estimate}}{\leq} C \varepsilon^{-\alpha + \beta + \mu }\nonumber\\
    &\stackrel{\mu > \alpha}{=} o(\varepsilon^{\beta}) \label{localmin vorticity estimate}.
    \end{align}
    An analogous argument shows that item (2) holds, that is, $$\norm{\mu(u,A)}{\left(C^{0,1}_0(\Omega)\right)^*} = o(1).$$
    
    Let us now provide an upper bound for $|R_0|$. Combining $h_\ex\leq \ep^{-\alpha}$ with \eqref{localmin energy bound}, \eqref{xi lipschitz estimate} and \eqref{boundR0} yields
    \begin{align*}
        |R_0| &\leq C h_{\ex}^2 \varepsilon \fen(u,A)^{\frac{1}{2}}\norm{\nabla \xfield}{L^\infty(\Omega)} \\
        &\leq C\varepsilon^{\frac{\beta}{2} +1 - 2\alpha}.
    \end{align*}
    Observe that since $\beta < 2-4\alpha$, we have
    $$\frac{\beta}{2} + 1 - 2\alpha > \frac{\beta}{2} + \frac{\beta}{2} = \beta,$$
    which means that
    \begin{equation}\label{localmin r0 bound}
        |R_0| = o(\ep^{\beta}).
    \end{equation}
    Therefore, inserting \eqref{localmin vorticity estimate} and \eqref{localmin r0 bound} into \eqref{boundforenergy}, we deduce that
    \begin{equation}\label{localmin free energy in U}
        \fen(u,A) \leq h_\ex\left| \int_\Omega \mu(u,A)\xfield\right| + |R_0| \leq   o(\varepsilon^\beta).
    \end{equation}
    The claim is thus proved, that is, $(u,A) \in U$ for small enough $\ep$. Moreover, since $U$ is open, the configuration $(\mathbf{u,A})$ must be a local minimizer of $GL_\varepsilon$.

    Finally, by combining \eqref{localmin energy bound}, \eqref{localmin vorticity estimate} and \eqref{localmin r0 bound}, we conclude that item (1) holds, since (recall \eqref{up})
    $$GL_\ep(\pmin, h_\ex A^0_\ep) \geq GL_\varepsilon(\u,\A) = GL_\ep(\pmin, h_\ex A^0_\ep) + O(\ep^\beta).$$

    Let us now prove that $(\u,\A)$ is a vortexless configuration. Since $(\mathbf{u,A})$ is a local minimizer, it solves the Ginzburg--Landau equations \eqref{euler lagrange gl}. Since we have \eqref{localmin free energy in U}, it follows from \eqref{coulomb linfty estimate} that $\norm{A}{L^\infty(\Omega)} \leq \frac{C}{\ep}$ and therefore, from \eqref{coulomb grad estimate}, that 
    $$
    \norm{\nabla |u|}{L^\infty(\Omega)} \leq \norm{\nabla u}{L^\infty(\Omega)} \leq \frac{C}{\ep}.
    $$ 
    Thus, by Proposition \ref{clearing out}, item (3) holds.
    
    \item \emph{(Closeness to the Meissner configuration. Proof of items (4) and (5)).} We start by estimating $\norm{\nabla u}{L^2(\Omega)}$. Note that
    $$ \int_\Omega |\nabla u|^2 \leq 2\left(\int_\Omega |\nabla_A u|^2 + |A|^2 |u|^2 \right).$$ 
    On the other hand, the Coulomb gauge estimate \eqref{coulomb h1 estimate} yields
    \begin{equation}\label{after splitting A h1 estimate}
        \norm{A}{H^1(\Omega)} \leq C \norm{\curl A}{L^2(\Omega)} \leq C \fen(u,A)^{\frac{1}{2}} \stackrel{\eqref{localmin free energy in U}}{=} o(\varepsilon^{\frac{\beta}{2}}).
    \end{equation}
        This together with the uniform convergence from item (1) $\norm{1-|u|}{L^\infty(\Omega)} = o(1)$, leads us to
        \begin{equation}\label{after splitting grad u l2 estimate}
            \int_\Omega |\nabla u|^2 \leq C\left(\fen(u,A)  + \norm{A}{L^2(\Omega)}^2\right) =o(\varepsilon^\beta).
        \end{equation}   
        Let us now provide an estimate for $\norm{u}{L^2(\Omega)}$. Defining $\overline{u} = \frac{1}{|\Omega|} \int_\Omega u$, by the Poincar\'e--Wirtinger inequality, we have
        $$\int_\Omega |u-\overline{u}|^2 \leq C \int_\Omega |\nabla u|^2 = o(\varepsilon^\beta)$$
        We then deduce that
        \begin{align*}
            \int_\Omega (1-|\overline{u}|)^2 \leq 2 \left(\int_\Omega |1-u|^2  + |u-\overline{u}|^2\right) = o(\varepsilon^\beta).
        \end{align*}
        Since $\overline{u}$ is constant in $\Omega$, we deduce that $u=e^{i\theta_\varepsilon} + o(\varepsilon^\frac{\beta}{2})$. Combining this with \eqref{after splitting A h1 estimate} and \eqref{after splitting grad u l2 estimate} yields that item (4) holds, since 
        \begin{equation}\label{after splitting approximation}
            \inf_{\theta \in [0,2\pi]} \norm{u-e^{i\theta}}{H^1(\Omega)} + \norm{A}{H^1(\Omega)} = \inf_{\theta \in [0,2\pi]} \norm{u-e^{i\theta}}{L^2(\Omega)} + \norm{\nabla u}{L^2(\Omega)} + \norm{A}{H^1(\Omega)} = o(\varepsilon^\frac{\beta}{2}).
        \end{equation}

        Finally, we prove item (5). The estimate on $\norm{\mathbf{A} - h_{\ex}A^0_\ep}{H^1(\Omega)}$ follows immediately, since $\mathbf{A} - h_{\ex} A^0_\ep = A$.
        On the other hand, for $r \in [1,2)$, let $s>2$ such that $\frac{1}{r} = \frac{1}{s} + \frac{1}{2}$. Then, using Hölder's inequality and a Sobolev embedding, we deduce that (recall $\u=\pmin u$) 
        \begin{align*}
            \norm{\mathbf{u}-\pmin e^{i\theta}}{W^{1,r}(\Omega)} &\leq \norm{\pmin(u-e^{i\theta})}{L^r(\Omega)} + \norm{\pmin \nabla u}{L^r(\Omega)} + \norm{(u-e^{i\theta})\nabla \pmin}{L^r(\Omega)} \\
            &\stackrel{\pmin\leq 1}{\leq} C\left(\norm{u-e^{i\theta}}{L^2(\Omega)} + \norm{\nabla u}{L^2(\Omega)}\right) + \norm{u-e^{i\theta}}{L^s(\Omega)}\norm{\nabla \pmin}{L^2(\Omega)} \\
            &\leq C\norm{u-e^{i\theta}}{H^1(\Omega)}\left(1+\norm{\nabla \pmin}{L^2(\Omega)}\right)\\
            &\stackrel{\eqref{after splitting approximation}}{\leq} C\norm{u-e^{i\theta}}{H^1(\Omega)}(1+\varepsilon^{-\gamma}),
        \end{align*}
        where in the last inequality we used the hypothesis $\norm{\nabla \pmin}{L^2(\Omega)} \leq \ep^{-\gamma}$, for $\gamma < 1-2 \alpha$. Since, until this point, the choice of $\beta\in(0, 2-4\alpha)$ was arbitrary, we may change it if necessary, so that $\frac{\beta}{2} \in (\gamma,1-2\alpha)$. Hence
        $$\inf_{\theta \in [0,2\pi]} \norm{\u - \pmin e^{i\theta}}{W^{1,r}(\Omega)} \leq C \ep^{-\gamma} \inf_{\theta \in [0,2\pi]} \norm{u-e^{i \theta}}{H^1(\Omega)} \stackrel{\eqref{after splitting approximation}}{\leq} o(\ep^{\frac{\beta}{2}})\ep^{-\gamma} = o(1).$$
        This concludes the proof.
    \end{enumerate}
\end{proof}

\subsection{Proof of Theorem \ref{uniqueness-above-fcf}}
We now prove the uniqueness (up to a gauge transformation) of a vortexless minimizing configuration.

\begin{proof}
    We will adapt the proofs of \cite{sylvia-arma}*{Section 2} and \cite{Rom-CMP}*{Theorem 1.5}. To prove uniqueness up to a gauge transformation, we will prove that there is a unique minimizer in the Coulomb gauge. Suppose $(\u_j, \A_j) = (\pmin u_j, A_j+ h_{\ex} A_\varepsilon^0)$ are distinct local minimizers, where $\A_j$ satisfies the Coulomb gauge condition \eqref{coulomb gauge} for $j=1,2$. Since $A^0_\ep$ also satisfies \eqref{coulomb gauge}, we deduce that $A_j$ does it as well.

    By \eqref{coulomb grad estimate} and Proposition \ref{clearing out}, we have that $|u_j|$ converges uniformly to 1. In particular, we have $|u_j| \geq \frac{3}{4}$ for small enough $\ep$\footnote{Actually, any $c$ in the domain of convexity of $(1-x^2)^2$ will do, that is $|u_j| \geq c > \frac{1}{\sqrt{3}}$. We choose $\frac34$ as in \cite{sylvia-arma}.}. Therefore we can write $u_j =\eta_j e^{i\phi_j}$, where $\eta_j = |u_j|$. Note that $(u_j,A_j)$ is gauge-equivalent to $(\eta_j, A'_j)$, where $A'_j = A_j - \nabla \phi_j$. Let $A^\circ_j = A_j +h_{\ex}A_\ep^0 - \nabla \phi_j$, which is gauge-equivalent to $(\u_j, \A_j)$ and therefore is a local minimizer.
    \begin{enumerate}[label=\textsc{\bf Step \arabic*},leftmargin=0pt,labelsep=*,itemindent=*,itemsep=10pt,topsep=10pt]
    \item \emph{(Proving that$\norm{A^\circ_j}{L^\infty(\Omega)} = o(\ep^{-1})$).} Observe that 
    $$\norm{A^\circ_j}{L^\infty(\Omega)} \leq \norm{A_j+ h_{\ex} A^0_{\ep}}{L^\infty(\Omega)} + \norm{\nabla \phi_j}{L^\infty(\Omega)}.$$      
    From \eqref{coulomb linfty estimate}, we have that 
    \begin{equation}\label{estimaAj}
    \norm{A_j + h_\ex A^0_\ep}{L^\infty(\Omega)} = o(\ep^{-1}),
    \end{equation}
    since we are assuming $E_\ep(\pmin) \ll \frac{1}{\ep^2}$ and $\fen(u_j,A_j) < \ep^\beta$. We are then left to prove $\norm{\nabla \phi_j}{L^\infty(\Omega)} = o(\ep^{-1})$. 
    
    By gauge-invariance, $(\pmin \eta_j,A^\circ_j)$ is also a local minimizer and thus, it satisfies \eqref{euler lagrange gl}. In particular, we have
    $$-\nabla^\perp \curl A_j^\circ = \ip{i \pmin \eta_j}{\nabla_{A^\circ_j} (\pmin \eta_j)} = -(\pmin \eta_j)^2 A^\circ_j\quad \mathrm{in}\ \Omega,$$
    which implies that
    $$
    \diver \left(\pmin^2 \eta_j^2 A_j^\circ\right) = \diver\left(\pmin^2 \eta_j^2(A + h_{\ex} A^0_\ep - \nabla \phi_j)\right) = \diver\left(\nabla^\perp \curl A^\circ_j\right) = 0 \quad \mathrm{in}\ \Omega.
    $$ Moreover, since $\pmin^2 A^0_\ep = -\nabla^\perp \xfield$, we have that $\diver \left(\pmin^2 A^0_\ep\right) = 0$ in $\Omega$. Recalling that $A_j$ satisfies \eqref{coulomb gauge}, a direct calculation then yields
    $$
    2\pmin^2 \eta \nabla \eta_j \cdot A_j^\circ + 2 \eta_j^2 \pmin \nabla \pmin \cdot A'_j - \eta_j^2 \pmin^2 \Delta \phi_j = 0\quad \mathrm{in}\ \Omega.
    $$
    
    On the other hand, from the first boundary condition in \eqref{euler lagrange gl}, we have that 
    $$
    \nabla_{A^\circ_j} (\pmin \eta_j) \cdot \nu =0 \quad \mathrm{on}\ \partial \Omega.
    $$
    Recalling the boundary condition in \eqref{pde rho} and that both $A$ and $A^0_\ep$ satisfy \eqref{coulomb gauge}, we deduce that 
    $$
    \left(\nabla \eta_j - i \nabla \phi_j\right) \cdot \nu = 0 \quad \mathrm{on}\ \partial \Omega
    $$ 
    and, in particular, $\nabla \phi_j \cdot \nu = 0$ on $\partial\Omega$. 
    
    Hence, $\phi_j$ solves the following elliptic PDE
    \begin{equation}\label{phi equation}
            \left\{
            \begin{array}{rcll}
            -\Delta \phi_j &=& -2\left(\dfrac{\nabla \eta_j}{\eta_j} \cdot A^\circ_j + \dfrac{\nabla \pmin}{\pmin} \cdot A'_j\right)  &\mathrm{in}\ \Omega\\
            \dfrac{\partial \phi_j}{\partial \nu} &=& 0&\mathrm{on}\ \partial \Omega.
            \end{array}
            \right.
        \end{equation}
    Since $\eta_j \geq \frac34 > 0$ and $\pmin \geq \sqrt{b} > 0$, we have, for any $q>1$, that
    $$
    \norm{\Delta \phi_j}{L^q(\Omega)} \leq C\left(\norm{\nabla \eta_j \cdot A^\circ_j}{L^q(\Omega)} + \norm{\nabla \pmin \cdot A'_j}{L^q(\Omega)})\right.
    $$
    We now estimate the terms in the RHS by interpolating between $L^2(\Omega)$ and $L^\infty(\Omega)$. The $L^\infty$-bounds come from our estimates for critical points of $GL_\ep$ in the Coulomb gauge, whereas the $L^2$-bounds follow from the smallness of $\fen(u_j, A_j)$, since we have
    \begin{equation}\label{free energy hypothesis}
        \fen(u_j,A_j) = \frac{1}{2} \int_\Omega \pmin^2\left(|\nabla \eta_j|^2 + |\eta_j|^2|A_j-\nabla \phi_j|^2\right) + |\curl A_j|^2 + \pmin^4 \frac{(1-\eta^2)^2}{2\ep^2} < \ep^\beta.
    \end{equation}
    First, for any $q>2$, we have
    \begin{align}\label{estimaetaj}
        \norm{\nabla \eta_j}{L^q(\Omega)} &\leq \norm{\nabla \eta_j}{L^\infty(\Omega)}^{1-\frac{2}{q}} \norm{\nabla \eta_j}{L^2(\Omega)}^\frac{2}{q}\notag \\
        &\stackrel{\eqref{coulomb grad estimate}}{\leq} C(\varepsilon^{-1})^{1-\frac{2}{q}} \left(\frac{1}{\|\rho_\ep\|_{L^\infty(\Omega)}}\fen(u_j,A_j) \right)^{\frac{1}{q}}\notag\\
        &\stackrel{\eqref{free energy hypothesis}}{\leq} C\varepsilon^{\frac{2}{q}-1} \varepsilon^{\frac{\beta}{q}}=C\ep^{\frac{2+\beta}{q} - 1}.
    \end{align}
    Second, for any $q>2$, we have
    \begin{align}\label{estimaphij}
        \norm{\nabla \phi_j}{L^q(\Omega)} &\leq \norm{\nabla \phi_j}{L^\infty(\Omega)}^{1-\frac{2}{q}} \norm{\nabla \phi_j}{L^2(\Omega)}^\frac{2}{q} \notag\\
        &\stackrel{\eqref{coulomb grad estimate}}{\leq} C(\varepsilon^{-1})^{1-\frac{2}{q}} \left(\norm{\nabla \phi_j - A_j}{L^2(\Omega)} + \norm{A_j}{L^2(\Omega)}\right)^{\frac{2}{q}}\notag\\
        &\stackrel{\eqref{coulomb h1 estimate}}{\leq} C\ep^{\frac{2}{q}-1}\left(\frac{1}{\norm{\pmin \eta_j}{L^\infty(\Omega)}}\fen(u_j,A_j)^\frac{1}{2}+\norm{\curl A_j}{L^2(\Omega)} \right)^{\frac{2}{q}}\notag\\
        &\stackrel{\eqref{free energy hypothesis}}{\leq} C\ep^{\frac{2}{q}-1} \ep^{\frac{\beta}{q}}=C\ep^{\frac{2+\beta}{q} - 1}.
    \end{align}
    Hence, from \eqref{estimaAj}, \eqref{estimaetaj}, and \eqref{estimaphij}, we conclude that, for any $q\in (2,2+\beta)$, we have
    \begin{align*}
        \norm{\nabla \eta_j \cdot A^\circ_j}{L^q(\Omega)} &\leq  \norm{\nabla \eta_j \cdot (A_j + h_{\ex} A^0_\ep)}{L^q(\Omega)} + \norm{\nabla \eta_j \cdot \nabla \phi_j}{L^q(\Omega)}\\
        &\leq \norm{\nabla \eta_j}{L^q(\Omega)} \norm{A_j + h_{\ex}A^0_\ep}{L^\infty(\Omega)} + \norm{\nabla \eta_j}{L^q(\Omega)}\norm{\nabla \phi_j}{L^\infty(\Omega)}\\
        &\stackrel{\eqref{coulomb linfty estimate}}{\leq} o(\ep^{-1})
    \end{align*}
    and 
    \begin{align*}
        \norm{\nabla \pmin \cdot A'_j}{L^q(\Omega)} &\leq \norm{\nabla \pmin}{L^\infty(\Omega)} \norm{A_j}{L^q(\Omega)} + \norm{\nabla \pmin}{L^\infty(\Omega)} \norm{\nabla \phi_j}{L^q(\Omega)} \\
        &\leq \frac{C}{\ep}\norm{A_j}{H^1(\Omega)} +o(\ep^{-1})\\
        &\stackrel{\eqref{coulomb h1 estimate}}{\leq} \frac{C}{\ep} \norm{\curl A_j}{L^2(\Omega)}+o(\ep^{-1})\\
        &\stackrel{\eqref{free energy hypothesis}}{\leq} o(\ep^{-1}),
    \end{align*}
    where after the first inequality we used Proposition \ref{proprho_ep} and Sobolev embedding.
    
    It follows that
    $$
    \norm{\Delta \phi_j}{L^q(\Omega)} = o(\ep^{-1})
    $$
    and, since $q>2$, by elliptic regularity and a Sobolev embedding, we have
    $$\norm{\nabla \phi_j}{L^{\infty}(\Omega)} = o(\ep^{-1}).$$
    This finally yields that 
    $$
    \norm{A^\circ_j}{L^\infty(\Omega)} \leq \norm{A_j+h_{\ex}A^0_\ep}{L^\infty(\Omega)} + \norm{\nabla \phi_j}{L^\infty(\Omega)} = o(\ep^{-1}).
    $$
    \item \emph{(Convexity argument)} By gauge-invariance, we have
    $$
    GL_\varepsilon(\u_j,\A_j) = GL_\varepsilon(\pmin \eta_j, A^\circ_j) = E_\ep(\pmin \eta_j,A^\circ_j)+\frac12\int_\Omega |\curl A^\circ_j - h_{\ex}|^2.
    $$
    Using \eqref{LM eq}, we have
    $$
    GL_\ep(\pmin \eta_j, A^\circ_j) = E_\ep(\pmin) + \frac{1}{2} \int_\Omega \pmin^2 \left(|\nabla \eta_j|^2 + \eta_j^2|A^\circ_j|^2\right)  + \pmin^4 \frac{(1-\eta_j^2)^2}{2\ep^2}+ |\curl A^\circ_j - h_\ex|^2.
    $$
    Let us define 
    $$
    Y \colonequals  \frac{GL_\varepsilon(\pmin \eta_1, A^\circ_1) + GL_\varepsilon(\pmin \eta_2, A^\circ_2)}{2} - GL_\varepsilon \left( \pmin\frac{\eta_1+ \eta_2}{2}, \frac{A^\circ_1 + A^\circ_2}{2}\right).
    $$
    We claim that $Y>0$. To prove this, let us write $Y =  \frac{1}{2}(Y_1 + Y_2 + Y_3)$, where
    \begin{align*}
        Y_1 &= \left(\int_\Omega \pmin^2\left(\frac{|\nabla \eta_1|^2 + |\nabla \eta_2|^2}{2}\right) - \int_\Omega \pmin^2 \left| \nabla \left(\frac{\eta_1+\eta_2}{2}\right) \right|^2 \right)  \\
        &\ \ +\left( \int_\Omega \frac{|\curl A^\circ_1 - h_{\ex}|^2 + |\curl A^\circ_2 - h_{\ex}|^2}{2} - \int_\Omega \left| \curl \left( \frac{A^\circ_1+A^\circ_2}{2}\right) - h_{\ex} \right|^2 \right),\\
        Y_2 &= \int_\Omega \frac{\pmin^4}{2\ep^2} \left(\frac{(1-\eta_1^2)^2+(1-\eta_2^2)^2}{2}\right) - \int_\Omega \frac{\pmin^4}{2\ep^2}\left(1 - \left(\frac{\eta_1 + \eta_2}{2} \right)^2  \right)^2,\ \mathrm{and}\\
        Y_3 &= \int_\Omega \pmin^2\left( \frac{|A^\circ_1|^2|\eta_1|^2 + |A^\circ_2|^2|\eta_2|^2}{2}\right) - \int_\Omega \pmin^2\left(\left|\frac{\eta_1+\eta_2}{2} \right|^2 \left|\frac{A^\circ_1+ A^\circ_2}{2} \right|^2\right).
    \end{align*}
    By convexity, we have that $Y_1 \geq 0$.
    
    Let us now provide an estimate for $Y_2$. A direct calculation yields (see \cite{sylvia-arma}*{Section 2} for the details)
    $$\frac{(1-\eta_1^2)^2+(1-\eta_2^2)^2}{2} - \left(1-\left(\frac{\eta_1+\eta_2}{2} \right)^2\right)^2 = \frac{1}{16}(\eta_1 - \eta_2)^2(7(\eta_1 + \eta_2)^2 - 4 \eta_1 \eta_2 - 8).$$
    Therefore, we have
    $$Y_2 = \frac{1}{32 \ep^2} \int_\Omega \pmin^4(\eta_1-\eta_2)^2(7(\eta_1+\eta_2)^2 - 4\eta_1 \eta_2 - 8),$$
    which combined with $\frac34 \leq \eta_j \leq 1$ and $\rho_\ep^4\geq b^2$, yields
    \begin{equation}\label{Y2}
    Y_2 \geq \frac{b^2}{32 \ep^2}\left(7 \left(\frac34 + \frac34 \right)^2 - 12\right) \norm{\eta_1-\eta_2}{L^2(\Omega)}^2=\frac{C_1}{\ep^2}\norm{\eta_1-\eta_2}{L^2(\Omega)}^2,
    \end{equation}
    where $C_1>0$ is a constant that depends on $b$ only.

    Let us now estimate $Y_3$. A direct calculation shows that (see \cite{sylvia-arma}*{Section 2} for the details)
    \begin{multline*}
        \frac{|A^\circ_1|^2|\eta_1|^2 + |A^\circ_2|^2|\eta_2|^2}{2} - \left|\frac{\eta_1+\eta_2}{2} \right|^2 \left|\frac{A^\circ_1+ A^\circ_2}{2} \right|^2 \\
        = \frac18(\eta_1 - \eta_2)^2|A^\circ_1 + A^\circ_2|^2 +\frac12\eta_1^2 |A^\circ_1 - A^\circ_2|^2 \\
        -\frac18(\eta_1 - \eta_2)(A^\circ_1 - A^\circ_2) \left(A^\circ_1(2 \eta_1 + 4 \eta_2) + A^\circ_2(6 \eta_1 + 8 \eta_2)\right).
    \end{multline*}
    Therefore
    \begin{multline*}
        Y_3 = \frac{1}{8} \int_\Omega \pmin^2 \big((\eta_1 - \eta_2)^2|A^\circ_1 + A^\circ_2|^2 +4\eta_1^2 |A^\circ_1 - A^\circ_2|^2\big) \\
    - \frac18\int_\Omega \pmin^2 \big((\eta_1 - \eta_2)(A^\circ_1 - A^\circ_2) (A^\circ_1(2 \eta_1 + 4 \eta_2) + A^\circ_2(6 \eta_1 + 8 \eta_2)\big)
    \end{multline*}
    which combined with $\pmin \eta_j \leq 1$, yields
    \begin{multline}\label{Y3}
    Y_3 \geq \frac{1}{8} \int_\Omega \pmin^2 \big((\eta_1 - \eta_2)^2|A^\circ_1 + A^\circ_2|^2 +4\eta_1^2 |A^\circ_1 - A^\circ_2|^2\big) \\
    -\frac18\int_\Omega \pmin |\eta_1 - \eta_2||A^\circ_1 - A^\circ_2| (6|A^\circ_1| + 14|A^\circ_2|).
    \end{multline}
    Note that $Y_3 \geq 0$ if $\eta_1 \equiv \eta_2$ or $A^\circ_1 \equiv A^\circ_2$, which in turn yields that $Y>0$. Indeed, if $\eta_1\equiv \eta_2$, $A_1^\circ\not \equiv A_2^\circ$, we have $Y_3 > \frac{1}{2} \int_\Omega \eta_1^2 |A^\circ_1 - A^\circ_2|^2 >0$. Hence, $Y\geq \frac12 Y_3>0$. On the other hand, if $A_1^\circ\equiv A_2^\circ$, then $\eta_1\not\equiv  \eta_2$, and therefore $Y\geq \frac12 Y_2>0$. 
    For this reason, we assume from now on that $\eta_1 \not\equiv  \eta_2$ and $A^\circ_1 \not\equiv A^\circ_2$. 
    
    From the $L^\infty$-bound obtained in {\bf Step 1}, we deduce that
    \begin{align}\label{Y4}
        \int_\Omega \pmin |\eta_1-\eta_2|&|A^\circ_1 - A^\circ_2| (6|A^\circ_1| + 14|A^\circ_2|)\notag \\
        &\leq \norm{\eta_1 - \eta_2}{L^2(\Omega)} \norm{A^\circ_1 - A^\circ_2}{L^2(\Omega)}\big(14(\norm{A^\circ_1}{L^\infty(\Omega)} + \norm{A^\circ_2}{L^\infty(\Omega)})\big)\notag \\
        &\leq o(\ep^{-1})\norm{\eta_1 - \eta_2}{L^2(\Omega)} \norm{A^\circ_1 - A^\circ_2}{L^2(\Omega)}.
    \end{align}
    On the other hand, using once again that $\eta_1\geq \frac34$ and $\rho_\ep^2\geq b$, from Young's inequality, we deduce that
    \begin{equation}\label{Y5}
    \frac12\int_\Omega \pmin^2 \eta_1^2|A^\circ_1 - A^\circ_2|^2 + \frac{C_1}{\ep^2} \norm{\eta_1 - \eta_2}{L^2(\Omega)}^2 \geq \frac{C_2}{\ep} \norm{\eta_1 - \eta_2}{L^2(\Omega)} \norm{A^\circ_1 - A^\circ_2}{L^2(\Omega)},
    \end{equation}
    where $C_2>0$ is a constant that depends on $b$ only. Finally, by combining \eqref{Y2}, \eqref{Y3}, \eqref{Y4}, and \eqref{Y5}, we are led to
    \begin{align*}
        Y_2 + Y_3 \geq \norm{\eta_1 - \eta_2}{L^2(\Omega)} \norm{A^\circ_1 - A^\circ_2}{L^2(\Omega)}\left(\frac{C}{\ep} - o(\ep^{-1}) \right).
    \end{align*}
    Hence, for sufficiently small $\ep$, we have $Y > 0$ on all cases.
    
    \item \emph{(Contradiction)} Assume without loss of generality that 
    $$GL_\varepsilon(\pmin \eta_1, A^\circ_1) \leq GL_\varepsilon(\pmin \eta_2, A^\circ_2).$$ 
    Since $Y > 0$, we have 
    $$
    GL_\ep(\pmin \eta_2, A^\circ_2) \geq \frac{GL_\varepsilon(\pmin \eta_1, A^\circ_1) + GL_\varepsilon(\pmin \eta_2, A^\circ_2)}{2} > GL_\varepsilon \left(\pmin \frac{\eta_1+\eta_2}{2}, \frac{A^\circ_1 + A^\circ_2}{2}\right).
    $$
    A standard argument then shows that, for any $t\in (0,1)$, we have
    $$
    GL_\ep(\pmin \eta_2, A^\circ_2) > GL_\varepsilon \Big(\pmin (t\eta_1+(1-t)\eta_2), tA^\circ_1 + (1-t)A^\circ_2\Big),
    $$
    which contradicts the local minimality of $(\pmin \eta_2, A^\circ_2)$. Therefore, $(\u_1, \A_1) = (\u_2, \A_2)$, which concludes the proof.
\end{enumerate}
\end{proof}

\appendix
\section{Lower bound for a weighted free energy functional}\label{ball construction chapter}
    Given a ball $B \subset \Omega$ and a function $\eta_\varepsilon \colon B \to [\sqrt{b},1]$, we define
    $$F_{\varepsilon,\eta_\varepsilon,B} (u,A) \colonequals \frac{1}{2} \int_B \eta^2_\varepsilon |\nabla_A u|^2 + \eta_\varepsilon^4 \frac{(1-|u|^2)^2}{2\varepsilon^2}+|\curl A|^2.$$
    In this appendix we find a lower bound for the weighted free energy functional \eqref{free-energy-weight}, based on lower bounds for $F_{\ep,\eta_\ep,B}$ on suitable disjoint balls $B$ that cover the ``bad set'' $\left\{|u|\leq \frac12\right\}$. This corresponds to a slightly modified version of Jerrard's ball construction method \cite{jerrard}. More precisely, we will closely follow the refined ball construction method provided by Sandier and Serfaty in \cite{mass-displacement}, in which a lower bound is provided for each individual ball. The proofs are mostly the same, so we will go in detail only where the presence of the weight $\eta_\varepsilon$ makes a difference. In addition, the numbered constants $c_i, C_i$ will play the same role as in the proofs in \cite{mass-displacement}. 
    
    \medskip
    We start by obtaining a lower bound for an energy defined on a circle, which actually is the cornerstone of this new version of the ball construction method. 
    In the following, we use the notation $\underline{\eta^2_\varepsilon}(\Theta) \colonequals \min_{\Theta} \eta^2_\varepsilon$, for any closed subset $\Theta$ of $\Omega$.

    \begin{lemma}\label{cota frontera de bola}
         Let $r>0$ and $a \in \Omega$ such that $B=B(a,r) \in \Omega$. Define $m = \min_{\partial B}|u|$. Then, for any $\ep$ such that $0 < \frac{\ep}{\underline{\eta^2_\ep}(B(a,r)) } \leq r$, we have
         
        \begin{equation}\label{energy bound on boundary of ball}
            \frac{1}{2} \int_{\partial B(a,r)} \eta^2_\varepsilon |\nabla |u||^2 + \eta_\varepsilon^4 \frac{(1-|u|^2)^2}{2\varepsilon^2} \geq c_0 \underline{\eta^2_\varepsilon}(B) \frac{(1-m)^2}{\varepsilon},
        \end{equation}
        where $c_0$ is a universal constant.
    \end{lemma} 
    
    \begin{proof}
        We follow the proof of \cite{jerrard}*{Lemma 2.3}. Within this proof, $C$ denotes a positive constant that does not depend on $r$ and that may change from line to line.
        
        Let $x_m \in \partial B(a,r)$ such that $|u(x_m)| = m$ and define
        $$\gamma \colonequals \frac{1}{2} \int_{\partial B(a,r)} |\nabla|u||^2.$$
        From Morrey's inequality, we have, for any $x,y \in \partial B(a,r)$, that
        $$||u(x)|-|u(y)|| \leq C \norm{\nabla |u|}{L^2(\partial B(a,r))}|x-y|^{\frac{1}{2}} = C \gamma^{\frac{1}{2}} |x-y|^{\frac{1}{2}}.$$
        Therefore, for any $x \in \partial B(a,r)$ such that $|x-x_m|^\frac12 \leq \dfrac{|1-m|}{C\gamma^\frac12}$, we have 
        $$|u(x)| \leq |u(x_m)| + C\gamma^{\frac{1}{2}}|x-x_m|^{\frac{1}{2}} \leq  \frac{1+m}{2}.$$
        Since $r \geq \frac{\varepsilon}{\underline{\eta^2_\varepsilon}(B)}$, for any $\sigma>0$, the arclength of $\partial B(x,r) \cap B(x_m,\sigma)$ must be greater than $C \min\{\sigma,\frac{\varepsilon}{\underline{\eta^2_\varepsilon} (B)}\}$. 
        Moreover, since $(1-|u|^2)^2 \geq \frac{(1-m)^2}{C}$ whenever $|u| \leq \frac{1+m}{2}$, by choosing $\sigma = \frac{(1-m)^2}{\underline{\eta^2_\ep}(B) \gamma}$, we find
        \begin{align*}
            \frac{1}{2} \int_{\partial B(x,r)} \eta^2_\varepsilon |\nabla |u||^2 + & \eta^4_\ep \frac{(1-|u|^2)^2}{2\varepsilon^2} \\
            &\geq \underline{\eta^2_\ep}(B) \gamma + \underline{\eta^2_\varepsilon}(B)^2 \frac{(1-m)^2}{C \varepsilon^2} \min \left\{ \frac{\varepsilon}{\underline{\eta^2_\varepsilon}(B)}, \frac{(1-m)^2}{\underline{\eta^2_\ep}(B )\gamma} \right\}\\
            &= \underline{\eta^2_\ep}(B) \left(\gamma + \frac{(1-m)^2}{C\ep^2} \min\left\{\ep, \frac{(1-m^2)}{\gamma}\right\} \right).
        \end{align*}
        
        If $\ep \leq \frac{(1-m)^2}{\gamma}$, we obtain \eqref{energy bound on boundary of ball}. Otherwise, we can minimize $\gamma + \frac{K^2}{\gamma}$ with respect to $\gamma$, where $K = \frac{(1-m)^2}{C\varepsilon}$. Since $\gamma = K$ is a stationary point and $\gamma + \frac{K^2}{\gamma}$ is convex, we conclude that $2K$ is the minimum, which means $\gamma \geq 2K$. Therefore
        $$\frac{1}{2} \int_{\partial B(x,r)} \eta^2_\varepsilon |\nabla |u||^2 + \eta^4_\ep \frac{(1-|u|^2)^2}{2\varepsilon^2} \geq C \underline{\eta^2_\varepsilon}(B) \frac{(1-m)^2}{\varepsilon},$$
        which means $\eqref{energy bound on boundary of ball}$ holds in all cases.
    \end{proof}
    Recall the set $\Omega_\varepsilon = \{x \in \Omega \colon \operatorname{dist}(x,\partial \Omega) > \varepsilon\}$. Define $S=\{x \in \Omega_\ep \colon |u| \leq \frac{1}{2}\}$, and $S_E$ as the union of connected components $S_i$ of $\{|u| \leq 1/2\}$ with nonzero boundary degree. In addition, for a compact set $K \subseteq \Omega$ such that $\partial K \cap S_E = \emptyset$, we let
    $$\deg_E(u, \partial K) \colonequals \sum_{i} \deg(u,S_i).$$
    
    Applying the previous lemma, we obtain the following result.
    \begin{lemma}\label{initial-balls}
        There exists a (finite) 
        collection of disjoint closed balls $\{B_i\}_i = \{B(a_i,r_i)\}_i$ such that
        \begin{enumerate}[leftmargin=*,label={\normalfont (\arabic*)}]
            \item For each $i$, $r_i \geq \frac{\varepsilon}{\underline{\eta^2_\varepsilon}(B_i)}$. 
            \item $S_E \cap \Omega_\varepsilon \subseteq \cup_{i} B_i$.
            \item There exists a universal constant $c_1>0$ such that, for each $i$, we have 
            $$F_{\varepsilon, \eta_\varepsilon, \Omega \cap B_i}(u,A) \geq c_1\underline{\eta^2_\ep}(B) \frac{r_i}{\varepsilon}.$$
        \end{enumerate}
        \end{lemma}
    \begin{proof}
        The proof is a slight modification of the proof of \cite{jerrard}*{Proposition 3.3}. Indeed, by noting that from \cite{jerrard}*{Lemma 3.2}, we have
        $$\int_{S_i} \eta^2_\ep |\nabla u|^2\geq \underline{\eta^2_\ep}(S_i) \int_{S_i} |\nabla u|^2 \geq \frac{\underline{\eta^2_\ep}(S_i)}{C}|\deg(u,\partial S_i)|,$$
        the proof is exactly as the proof of \cite{jerrard}*{Lemma 3.3}, using of course \eqref{cota frontera de bola} instead of the lower bound in \cite{jerrard}*{Lemma 2.3} and the fact that $\underline{\eta^2_\ep}(\Theta_1)\geq \underline{\eta^2_\ep}(\Theta_2)$ for any closed sets such that $\Theta_1\subseteq \Theta_2$. The constant $c_1$ is the same as the constant $c_0$ in Jerrard's proof.
    \end{proof}
        From now on, we closely follow \cite{mass-displacement}*{Section 5}. 
    \begin{prop}
        For a small enough $c_2 \in (0,c_1)$, let $$\lambda_\varepsilon(x)=\min\left( \frac{c_2}{\varepsilon}, \frac{\pi}{x} \frac{1}{1+\frac{x}{2}+\frac{\pi \varepsilon}{c_0 x}} \right).$$
        Then, for any closed ball $B = B(a,r)$ such that $B \subset \Omega_\varepsilon$, $\partial B \cap S_E = \emptyset$, and $\frac{\varepsilon}{\underline{\eta^2_\varepsilon}(B)} \leq r \leq \frac{|d|}{2}$, where $d = \deg_{E}(u,\partial B) \neq 0$,  we have
        \begin{equation}\label{rd lower bound}
            \frac{1}{2} \int_{\partial B} \eta^2_\varepsilon |\nabla_A u|^2 + \eta^4_\ep \frac{(1-|u|^2)^2}{2\varepsilon^2}  + \frac{1}{2} \int_B |\operatorname{curl} A|^2  \geq \underline{\eta^2_\varepsilon}(B) \lambda_\varepsilon \left(\frac{r}{|d|} \right).
        \end{equation}
        Moreover, $\Lambda_\varepsilon(s) \colonequals \int_0^s \lambda_\varepsilon$  is increasing, the function $s \to \frac{\Lambda_\ep(s)}{s}$ is decreasing and it satisfies 
        $$\lim_{s \to 0} \frac{\Lambda_\ep(s)}{s} < \frac{c_1}{\ep}, \quad \frac{\Lambda_\ep(\ep)}{\ep}>\frac{c_3}{\ep},$$
        for some sufficiently small constant $c_3$. Finally, for any $s \in \left(\ep, \frac{1}{2}\right)$ and some $C_0>0$ we have
        $$\Lambda_\ep(s) \geq \pi \log \frac{s}{\ep} - C_0.$$
    \end{prop}

    \begin{proof}
        The proof is almost exactly as the proof of \cite{mass-displacement}*{Proposition 5.1}. In fact, the functions $\lambda_\ep, \Lambda_\ep$ are the same as in this proof, and since $\eta_\ep \leq 1$, we have $|\curl A|^2 \geq \underline{\eta^2_\ep}(B) |\curl A|^2$. Hence, we only need to carry around the weight $\underline{\eta^2_\ep}(B)$ and mimic the proof of \cite{mass-displacement}*{Proposition 5.1}.
    \end{proof} 

    With these estimates at hand, the ball construction procedure of growing and merging balls yields the following result.
    \begin{prop}\label{ball construction}
        For any $s\in\left(0,\frac12\right)$, there exists a collection of disjoint closed balls $\mathcal{B}(s)$, depending only on $u$, such that
        \begin{enumerate}[leftmargin=*,label={\normalfont (\arabic*)}]
            \item $\mathcal{B}(s) \subset \mathcal{B}(t)$ for $s<t$ and the total radius of the collection is continuous with respect to $s$.

            \item $S_E \subseteq \mathcal{B}(s)$, for any $s$.

            \item For any $B=B(a,r) \in \mathcal{B}(s)$, $$F_{\ep, \eta_\ep, B}(u,A) \geq \underline{\eta^2_\varepsilon}(B) r \frac{\Lambda_\varepsilon(s)}{s}.$$

            \item For any $B=B(a,r) \in \mathcal{B}(s)$ such that $B \subset \Omega_\varepsilon$, we have $r \geq s|d_B|$, where $d_B = \deg_E(u,\partial B)$.
        \end{enumerate}
    
    \end{prop}

    \begin{proof}
        The proof follows the process of growing and merging balls described in \cite{jerrard}*{Proposition 4.1} and \cite{mass-displacement}*{Proposition 5.2}. Let $\mathcal{B} = \{B_i\}_i = \{B(a_i,r_i)\}$ be the collection given by Lemma \ref{initial-balls}. We start by choosing $s_0<\frac{1}{2}$ small enough so that the balls in $\mathcal{B}$ satisfy items 3 and 4 (item 2 is obviously also satisfied). In particular, for each $B = B(a,r) \in \mathcal{B}$ we have
        $$F_{\ep, \eta_\ep, B}(u,A) \geq c_1 \underline{\eta^2_\varepsilon}(B) \frac{r}{\ep} \geq \underline{\eta^2_\varepsilon}(B) r \frac{\Lambda_\varepsilon(s_0)}{s_0}.$$ 
        We construct the collection $\mathcal{B}(s)$ as follows. For $s \leq s_0$, we let $\mathcal{B}(s) = \mathcal{B}$. Then, as $s$ increases, we let the radius of each ball grow so that $r_i = s|d_{B_i}|$. Observe that the bound of item 3 is preserved during the growth process, which follows from \eqref{rd lower bound} and the fact that $\underline{\eta^2_\ep}(B_i(s)) \geq \underline{\eta^2_\ep} (B_i(t))$ for $s < t$ (since $B_i(s)\subset B_i(t)$). If at a moment two balls $B_1 = B(a_1,r_1)$ and $B_2 = B(a_2,r_2)$ intersect each other, we merge these balls into a larger ball that contains them with a radius equal to the sum of the radii of the merged balls. This ball can be explicitly written as $B= B\left(\frac{a_1 r_1 + a_2 r_2}{r_1+r_2}, r_1+r_2\right)$. The bound of item 3 still holds after the merging process, since $|d_B| \leq |d_{B_1}| + |d_{B_2}|$ and $\underline{\eta^2_\ep}(B) \leq \underline{\eta^2_\ep}(B_1) + \underline{\eta^2_\ep}(B_2)$. This process of growing and merging continues as long as \eqref{rd lower bound} can be satisfied, that is, until $s = \frac{1}{2}$.
    \end{proof}
    Finally, we state our main energy estimate, which generalizes \cite{mass-displacement}*{Proposition 2.1} to the case of a weighted Ginzburg--Landau type energy.
    \begin{prop}\label{ball lower bound}
        There exist $\varepsilon_0, C > 0$ such that for any $\varepsilon < \varepsilon_0$ and $(u,A)$ such that
        $$\fen(u,A) \leq \varepsilon^{-\beta},$$
        where $\beta \in (0,1)$, the following holds. For every $r \in (C \varepsilon^{1-\beta},\frac{1}{2})$ there exists a collection of disjoint closed balls $\mathcal{B} = \{B_i\}_i = \{B(a_i,r_i)\}$ such that
        \begin{enumerate}[leftmargin=*,label={\normalfont (\arabic*)}]
            \item $\{x \in \Omega_\varepsilon \colon ||u|-1| \geq \frac{1}{2}\} \subseteq \cup_i B_i$. %
            \item $\sum_i r_i \leq r$.
            \item For any $2b \leq \overline{C}\leq \left(\frac{r}{\ep}\right)^{\frac{1}{2}}$ it holds that either 
            $$F_{\varepsilon, \eta_\varepsilon, \Omega \cap \mathcal{B}}(u,A) \geq \overline{C} \log \frac{r}{\varepsilon},$$
            or, for each $B \in \mathcal{B}$ such that $B \subset \Omega_\ep$,
            $$
            F_{\ep, \eta_\ep, B}(u,A) \geq \pi  \underline{\eta^2_\ep}(B)|d_B| \left( \log \frac{r}{\ep \overline{C} } - C \right),
            $$
            where $\underline{\eta^2_\ep}(B)=\min_B \eta^2_\ep$ and $d_B = \mathrm{deg}(u,\partial B)$. 
        \end{enumerate}
    \end{prop}
    \begin{proof} 
        The proof is exactly as the proof of \cite{mass-displacement}*{Proposition 2.1}. We only need to carry around the weight $\underline{\eta^2_\ep}(B)$ throughout the argument. 
    \end{proof}
    \begin{remark}
        Let us remark that \cite{mass-displacement}*{Proposition 2.1} states that the ball collection covers the set $\{x \in \Omega_\ep \colon |u(x)| < \frac{1}{2}\}$, contrary to what we have written here. However, a careful inspection of the proof reveals that the ball collection is obtained by merging with a cover of the set $\{x \in \Omega_\ep \colon |1-|u|| \geq \frac12\}$ given by \cite{libro-ss}*{Proposition 4.8}. This proposition also holds in the inhomogeneous case, since $b\leq \eta_\ep^2\leq 1$, which in turn gives $F_{\varepsilon, \eta_\varepsilon, \Omega }(u,A)\leq F_\ep(u,A)\leq b^{-1}F_{\varepsilon, \eta_\varepsilon, \Omega }(u,A)$.
    \end{remark}

    \begin{remark}
        In the situation where $d_B\neq 0$ for some $B\subset \Omega_\ep$, a natural choice for $\overline{C}$ is $\pi \tilde D$, where $\tilde D\colonequals \sum_{B\in \mathcal B\cap \Omega_\ep} \underline{\eta^2_\ep}(B)|d_B|$. With this choice, in all cases we have
            \begin{equation}\label{free energy lower bound equation}
                F_{\varepsilon, \eta_\varepsilon, \Omega \cap \mathcal{B}}(u,A) \geq \pi \Tilde{D} \left(\log \frac{r}{\ep \Tilde{D}} - C \right).
            \end{equation}
        Notice that this choice is possible since in this case $\bar C\geq \pi b>2b$. Moreover, if $d_B=0$ for every $B\subset \Omega_\ep$, then \eqref{free energy lower bound equation} still holds, since the RHS vanishes. Moreover, under the assumptions of Proposition \ref{ball lower bound}, we deduce from \eqref{free energy lower bound equation} and $r > C\ep^{1-\beta}$, that
        \begin{equation}\label{degree bound}
            \sum_i |d_{B_i}| \leq C\frac{\fen(u,A)}{\beta|\log \ep|},
        \end{equation}
        where $C>0$ is a constant that does not depend on $\ep$.
    \end{remark}

    \begin{remark}
        In \cite{mass-displacement}*{Proposition 2.1}, $\overline C$ must be larger than or equal to $2$. However, a careful inspection reveals that one can replace $2$ by any universal constant in $(0,\pi)$ and the argument of proof holds exactly the same. Notice that when $\eta_\ep\equiv 1$, $\pi\Tilde D\geq \pi$, and therefore we need to be able to choose $\overline C\geq \pi$ in order to obtain \eqref{free energy lower bound equation}. Of course, the condition $\overline C\geq 2$ makes this choice possible, but the same holds for any constant in $(0,\pi)$.
    \end{remark}
      
\bibliography{references2DPinning}
\end{document}